\documentclass[a4paper,11pt,reqno]{amsart}

\usepackage[margin=1in,includehead,includefoot]{geometry}
\usepackage[utf8]{inputenc}
\usepackage[T1]{fontenc}
\usepackage{lmodern,microtype}
\usepackage{upref}
\usepackage{mathtools,amssymb,amsthm}
\usepackage[foot]{amsaddr}
\usepackage{enumitem}
\usepackage{tikz}
\usetikzlibrary{shapes}
\usepackage{graphicx}
\usepackage{hyperref}
\usepackage{bookmark}
\usepackage{relsize}
\usepackage{tabularx}
\usepackage{mycommands}

\graphicspath{{./graphics/}}

\makeatletter
\let\old@setaddresses\@setaddresses
\def\@setaddresses{\bigskip{\parindent 0pt\let\scshape\relax\let\ttfamily\relax\old@setaddresses}}
\makeatother

\usepackage{mycommands}

\newtheorem{theorem}{Theorem}
\newtheorem{lemma}[theorem]{Lemma}
\newtheorem{corollary}[theorem]{Corollary}
\theoremstyle{remark}


\hypersetup{
  pdftitle={On flips in planar matchings}
  pdfauthor={Marcel Milich, Torsten M\"utze, Martin Pergel}
}

\title{On flips in planar matchings}

\author{Marcel Milich}
\address[Marcel Milich]{Institut f\"ur Mathematik, Technische Universit\"at Berlin, Germany}
\email{marcel.milich@campus.tu-berlin.de}

\author{Torsten M\"utze}
\address[Torsten M\"utze]{Department of Computer Science, University of Warwick, United Kingdom}
\email{torsten.mutze@warwick.ac.uk}

\author{Martin Pergel}
\address[Martin Pergel]{Department of Software and Computer Science Education, Charles University Prague, Czech Republic}
\email{perm@kam.mff.cuni.cz}

\thanks{Torsten M\"utze is also affiliated with the Faculty of Mathematics and Physics, Charles University Prague, Czech Republic, and he was supported by Czech Science Foundation grant GA~19-08554S and by German Science Foundation grant~413902284.
Martin Pergel was also supported by Czech Science Foundation grant GA~19-08554S}

\thanks{An extended abstract of this paper appeared as~\cite{proceedings:20}.}

\begin{document}

\begin{abstract}
In this paper we investigate the structure of flip graphs on non-crossing perfect matchings in the plane.
Specifically, consider all non-crossing straight-line perfect matchings on a set of $2n$ points that are placed equidistantly on the unit circle.
A \emph{flip} operation on such a matching replaces two matching edges that span an empty quadrilateral with the other two edges of the quadrilateral, and the flip is called \emph{centered} if the quadrilateral contains the center of the unit circle.
The graph~$\cG_n$ has those matchings as vertices, and an edge between any two matchings that differ in a flip, and it is known to have many interesting properties.
In this paper we focus on the spanning subgraph~$\cH_n$ of~$\cG_n$ obtained by taking all edges that correspond to centered flips, omitting edges that correspond to non-centered flips.
We show that the graph~$\cH_n$ is connected for odd~$n$, but has exponentially many small connected components for even~$n$, which we characterize and count via Catalan and generalized Narayana numbers.
For odd $n$, we also prove that the diameter of~$\cH_n$ is linear in~$n$.
Furthermore, we determine the minimum and maximum degree of~$\cH_n$ for all~$n$, and characterize and count the corresponding vertices.
Our results imply the non-existence of certain rainbow cycles in~$\cG_n$, and they resolve several open questions and conjectures raised in a recent paper by Felsner, Kleist, M\"utze, and Sering.
\end{abstract}


\maketitle

\section{Introduction}

Flip graphs are a powerful tool to study different classes of basic combinatorial objects, such as binary strings, permutations, partitions, triangulations, matchings, spanning trees etc.
A flip graph has as vertex set all the combinatorial objects of interest, and an edge between any two objects that differ only by a small local change operation called a \emph{flip}.
It thus equips the underlying objects with a structure that reveals interesting properties about the objects, and that allows one to solve different fundamental algorithmic tasks for them.

A classical example in the geometric context is the flip graph of triangulations, which has as vertices all triangulations of a convex $n$-gon, and an edge between any two triangulations that differ in removing the diagonal of some quadrilateral formed by two triangles and replacing it by the other diagonal.
A problem that has received a lot of attention is to determine the diameter of this flip graph, i.e., how many flips are always sufficient to transform two triangulations into each other.
This question was first considered by Sleator, Tarjan, and Thurston~\cite{MR928904} in the 1980s, and answered conclusively only recently by Pournin~\cite{MR3197650}.
This problem has also been studied extensively from an algorithmic point of view, with the goal of computing a short flip sequence between two given triangulations~\cite{DBLP:conf/cocoon/LiZ98,MR1744193,MR2541971,MR2740180}.
It is a well-known open question whether computing a shortest flip sequence is NP-hard.
Interestingly, the problem is known to be hard when the convexity assumption is dropped:
Specifically, computing a shortest flip sequence is NP-hard for triangulations of a simple polygon~\cite{MR3372115} and for general point sets~\cite{MR3399985}.

Another important property of flip graphs is whether they have a Hamilton path or cycle.
The reason is that computing such a path corresponds to an algorithm that exhaustively generates the underlying combinatorial objects~\cite{MR3444818}.
It is known that the flip graph of triangulations mentioned before has a Hamilton cycle~\cite{MR1723053}, and that a Hamilton path in this graph can be computed efficiently~\cite{MR1239499}.

Flip graphs also have deep connections to lattices and polytopes~\cite{MR3221544,MR3645055,MR3964495,DBLP:conf/wg/AichholzerCHKMS19}.
For instance, the aforementioned flip graph of triangulations of a convex $n$-gon arises as the cover graph of the well-known Tamari lattice, and can be realized as an $(n-3)$-dimensional polytope in several different ways~\cite{MR1022776,MR3437894}.
Other properties of interest that have been investigated for the flip graph of triangulations are its automorphism group~\cite{MR1022776}, the vertex-connectivity~\cite{MR1723053}, the chromatic number~\cite{MR2535071,berry-et-al:18}, its genus~\cite{MR3738334}, and the eccentricities of vertices~\cite{MR3994734}.
Similar results are known for flip graphs of other geometric configurations, such as matchings, spanning trees, partitions and dissections, etc.; see e.g.~\cite{MR2190792,MR2346418,MR2510231,MR3336579}.
Note that even for a fixed set of geometric objects such as matchings, there may be several natural notions of flips, yielding different flip graphs (cf.~\cite{MR2190792,MR2346418}).

\begin{figure}
\includegraphics{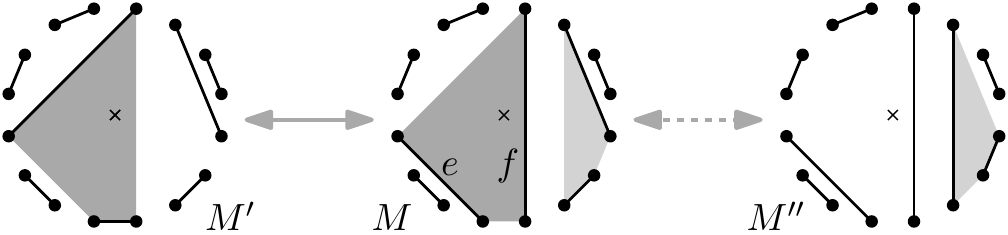}
\caption{Flips between matchings $M$ and $M'$ (centered flip), and between $M$ and $M''$ (non-centered flip).} 
\label{fig:flip}
\end{figure}

\begin{figure}
\includegraphics{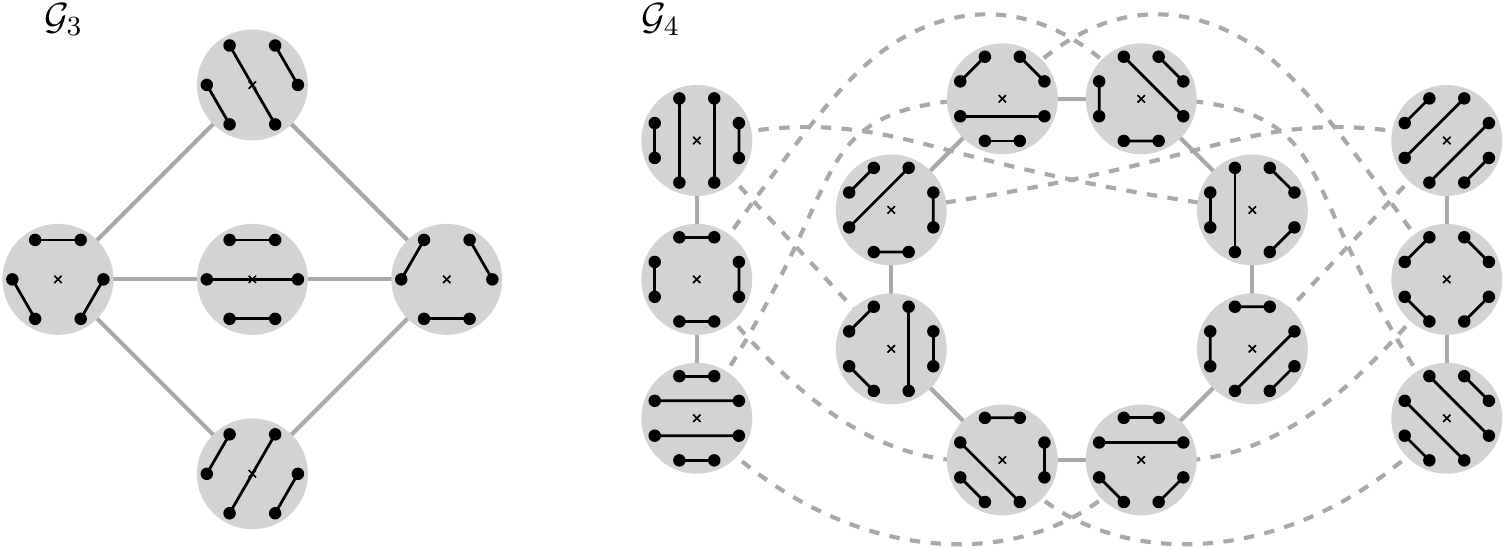}
\caption{Flip graphs $\cG_3$ (left) and $\cG_4$ (right).
Solid edges correspond to centered flips and are present in the subgraphs $\cH_3\seq\cG_3$ and $\cH_4\seq\cG_4$, whereas the dashed edges correspond to non-centered flips and are not present in these subgraphs.
} 
\label{fig:g34}
\end{figure}

\begin{figure}
\includegraphics{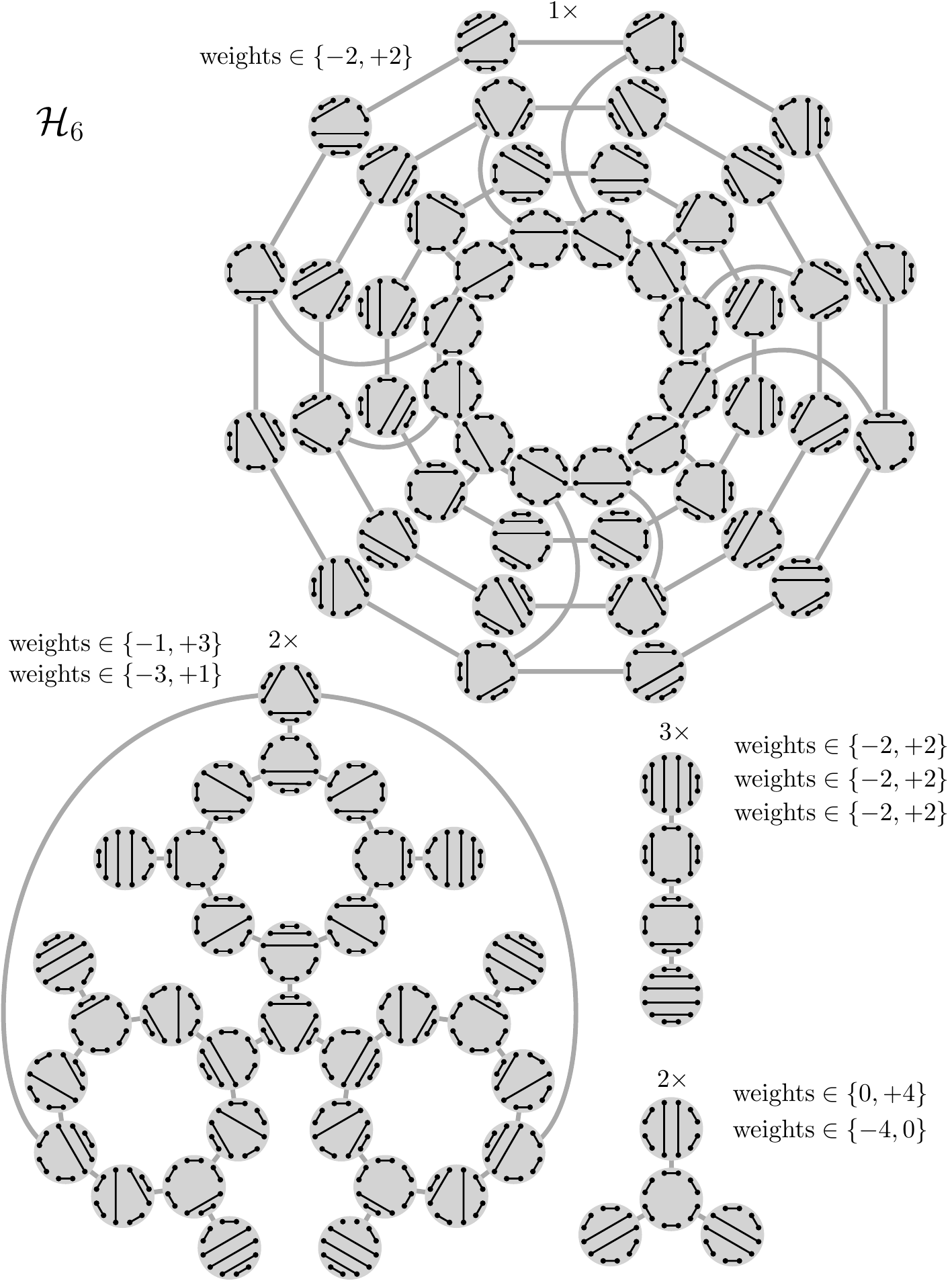}
\caption{The graph $\cH_6$ with the weights of all of its components.
Among the components, isomorphic copies are omitted, and the multiplicities are shown above the components.
The isomorpic copies differ only by rotation of the matchings.
For instance, there are two copies of the component shown at the bottom right.
}
\label{fig:h6}
\end{figure}

In this paper, we consider the flip graph of non-crossing perfect matchings in the plane defined as follows.
For any integer $n\geq 2$, we consider a set of $2n$~points placed equidistantly on a unit circle.
We let $\cM_n$ denote the set of all non-crossing straight-line perfect matchings with $n$ edges on this point set.
It is well-known that the cardinality of~$\cM_n$ is given by the $n$th \emph{Catalan number $C_n=\frac{1}{n+1}\binom{2n}{n}$}; see~\cite{MR3467982}.
For any matching $M\in\cM_n$, consider two matching edges $e,f\in M$ that span an empty quadrilateral, i.e., the convex hull of these two edges does not contain any other edges of~$M$; see Figure~\ref{fig:flip}.
Replacing the two edges $e$ and $f$ by the other two edges of the quadrilateral yields another matching $M'\in\cM_n$, and we say that $M$ and~$M'$ differ in a \emph{flip}.
The flip graph~$\cG_n$ has $\cM_n$ as its vertex set, and an undirected edge between any two matchings that differ in a flip.
Figure~\ref{fig:g34} shows the graphs~$\cG_3$ and~$\cG_4$.
Hernando, Hurtado, and Noy~\cite{MR1939072} proved that the graph~$\cG_n$ has diameter~$n-1$ and connectivity~$n-1$, is bipartite for all~$n$, has a Hamilton cycle for all even~$n\geq 4$, and no Hamilton cycle or path for any odd~$n\geq 5$.

We now distinguish between two different kinds of flips.
A flip is \emph{centered} if and only if the closed quadrilaterial that determines the flip contains the center of the unit circle.
For odd~$n$, the circle center may lie on the boundary of the quadrilateral, which still counts as a centered flip.
In Figure~\ref{fig:flip}, the flip between~$M$ and~$M'$ is centered, whereas the flip between~$M$ and~$M''$ is not.
In all our figures, the circle center is marked with a cross.
We let $\cH_n$ denote the spanning subgraph of~$\cG_n$ obtained by taking all edges that correspond to centered flips, omitting edges that correspond to non-centered flips.
Clearly, both graphs $\cH_n$ and $\cG_n$ have the same vertex set.
In Figure~\ref{fig:g34}, all solid edges belong to~$\cH_3$ or~$\cH_4$, respectively, whereas dashed edges do not.
The graph~$\cH_6$ is shown in Figure~\ref{fig:h6}.

The main motivation for considering centered flips comes from the study of rainbow cycles in flip graphs, a direction of research that was initiated in a recent paper by Felsner, Kleist, M\"utze, and Sering~\cite{MR4046775}.
A rainbow cycle generalizes the well-known concept of a balanced Gray code~\cite{MR1410880} to geometric flip graphs such as~$\cG_n$.
Balanced Gray codes are useful in applications that require some uniformity of the flip operations (see e.g.~\cite{MR2014514}), and in the context of the flip graph~$\cG_n$, this means intuitively that every flip operation occurs equally often along a cycle.
To be precise, an \emph{$r$-rainbow cycle} in the graph~$\cG_n$ is a cycle with the property that every possible matching edge appears exactly $r$ times in flips along this cycle.
Clearly, this means that every edge must also disappear exactly $r$ times in flips along the cycle.
In other words, if we assign each of the possible matching edges a distinct color, then we see every color appearing exactly $r$ times along the cycle, hence the name rainbow cycle.
For example, the graph~$\cG_4$ has a 1-rainbow cycle of length~8, formed by solid edges in Figure~\ref{fig:g34}, which all belong to the subgraph~$\cH_4$.
In fact, it was proved in~\cite[Lemma~10]{MR4046775} that every $r$-rainbow cycle in~$\cG_n$ may only contain centered flips, i.e., to search for rainbow cycles in~$\cG_n$, we may restrict our search to the much sparser subgraph~$\cH_n$.

Let us address another potential concern right away:
Our assumption that the $2n$ points of the point set are placed equidistantly on a unit circle is not necessary for expressing or proving any of our results.
Our results and proofs are indeed robust under moving the $2n$ points to any configuration in convex position, by suitably replacing all geometric notions by purely combinatorial ones.
In particular, centered flips can be defined without reference to the center of the unit circle (see Section~\ref{sec:prel}).
Nevertheless, in the rest of the paper we stick to the equidistancedness assumption, to be able to use both geometric and combinatorial arguments, whatever is more convenient.

\subsection{Our results}

In this work we investigate the structure of the graph~$\cH_n\seq\cG_n$.
In particular, we resolve some of the open questions and conjectures from the aforementioned paper~\cite{MR4046775} on rainbow cycles in flip graphs.
For several graph parameter, we observe an intriguing dichotomy between the cases where $n$ is odd or even, i.e., the graph~$\cH_n$ has an entirely different structure in those two cases.
Table~\ref{tab:res} contains a summary of our results, together with references to the theorems where they are established.

\begin{table}
\caption{Summary of properties of the graph~$\cH_n$.}
\label{tab:res}
\renewcommand{\arraystretch}{1.15}
\begin{tabular}{l|l|l}
graph property & odd $n\geq 3$ & even $n\geq 2$ \\ \hline
max.\ degree (Thm.~\ref{thm:max-deg}) & $n$ & $n/2$ \\
\# of max.\ deg.\ vertices & 2 & ? \\
min.\ degree (Thm.~\ref{thm:min-deg}) & 2 & 1 \\
\# of min.\ deg.\ vertices & $n\cdot(C_{(n-3)/2})^2$ & $n\cdot(C_{(n-2)/2})^2$ \\
diameter & \parbox[t][1cm][t]{3cm}{between $n-1$ and $11n-29$ (Thm.~\ref{thm:diam})} & $\infty$ \\
\# of components & \parbox[t][0cm][t]{3cm}{1 (Thm.~\ref{thm:conn})} & \parbox[t][2.5cm][t]{6cm}{${}\geq C_{n/2}+n-3$ (Thm.~\ref{thm:struct}+Cor.~\ref{cor:comp}); \\ $\binom{n}{n/2}$ many vertices form trees  \\ of size $n/2+1$ each; \\ component sizes bounded by Nara- \\ yana numbers (Thm.~\ref{thm:narayana} and Cor.~\ref{cor:size})} \\
$r$-rainbow cycles (Thm.~\ref{thm:rainbow}) & none for any $r\geq 1$ & none for large $r$ \\
Hamilton path/cycle & \multicolumn{2}{c}{none for $n\geq 4$\hphantom{xxxxxxxx}} \\
colorability & \multicolumn{2}{c}{bipartite (\cite{MR1939072})\hphantom{xxxxxxxx}}
\end{tabular}
\end{table}


Most importantly, the graph~$\cH_n$ is connected for odd~$n$ (Theorem~\ref{thm:conn}), but has exponentially many connected components for even~$n$ (Theorem~\ref{thm:struct} and Corollary~\ref{cor:comp}); cf.~Figure~\ref{fig:g34}.
For odd $n$, we show that the diameter of~$\cH_n$ is linear in~$n$ (Theorem~\ref{thm:diam}).
For even $n$, we provide a fine-grained picture of the component structure of the graph (Theorems~\ref{thm:struct} and~\ref{thm:narayana}, and Corollary~\ref{cor:size}).
We also describe the degrees of vertices in the graph~$\cH_n$ for all~$n$, and we characterize and count the vertices of minimum and maximum degree (Theorems~\ref{thm:deg}, \ref{thm:max-deg}, and~\ref{thm:min-deg}).
Finally, we easily see that the graph~$\cH_n$ does not admit a Hamilton cycle or path for any $n\geq 4$.
This follows from the non-Hamiltonicity of~$\cG_n$ for odd~$n\geq 5$ proved in~\cite{MR1939072}, and for even~$n\geq 4$ this is trivial as the graph~$\cH_n$ has more than one component.
Our results also imply the non-existence of certain rainbow cycles in the graph~$\cG_n$ (Theorem~\ref{thm:rainbow}).

In all of these results, Catalan numbers and generalized Narayana numbers make their appearance, and in our proofs we encounter several new bijections between different combinatorial objects counted by these numbers.

\subsection{Outline of this paper}

In Section~\ref{sec:prel} we discuss some terminology and observations that will be used throughout the paper.
In Section~\ref{sec:deg} we analyze the degrees of the graph~$\cH_n$. 
In Section~\ref{sec:odd} we present the structural results when the number~$n$ of matching edges is odd.
In Section~\ref{sec:even} we discuss the properties of~$\cH_n$ when $n$ is even.
Finally, in Section~\ref{sec:rainbow} we discuss the implications of our results with regards to the existence of rainbow cycles in~$\cG_n$.
We conclude with some open questions in Section~\ref{sec:open}.

\section{Preliminaries}
\label{sec:prel}

We first explain the combinatorial characterization of centered flips.
Given a matching $M\in\cM_n$, and any of its edges $e\in M$, we consider the number of other matching edges from~$M$ that lie on each of the two sides of the edge~$e$, and we let~$\ell(e)$ denote the minimum of these two numbers.
We refer to $\ell(e)$ as the \emph{length} of the edge~$e$.
If $\ell(e)=0$, then we refer to $e$ as a \emph{perimeter edge}.
In the example with $n=20$ matching edges shown in Figure~\ref{fig:edge}, the edge~$e$ has 7 other matching edges on one of its two sides, and 12 edges on the other side, and therefore $\ell(e)=\min\{7,12\}=7$.
Moreover, the matching in this figure has 11 perimeter edges.
We let $\mu=\mu(n)$ denote the maximum possible length of an edge, so the possible edge lengths are $0,1,\ldots,\mu$.
Clearly, we have
\begin{equation}
\label{eq:mu}
\mu=\lceil(n-2)/2\rceil=\begin{cases} (n-1)/2 & \text{ if $n$ is odd}, \\ (n-2)/2 & \text{ if $n$ is even}. \end{cases}
\end{equation}
The following lemma can be verified easily.

\begin{figure}
\makebox[0cm]{ 
\includegraphics{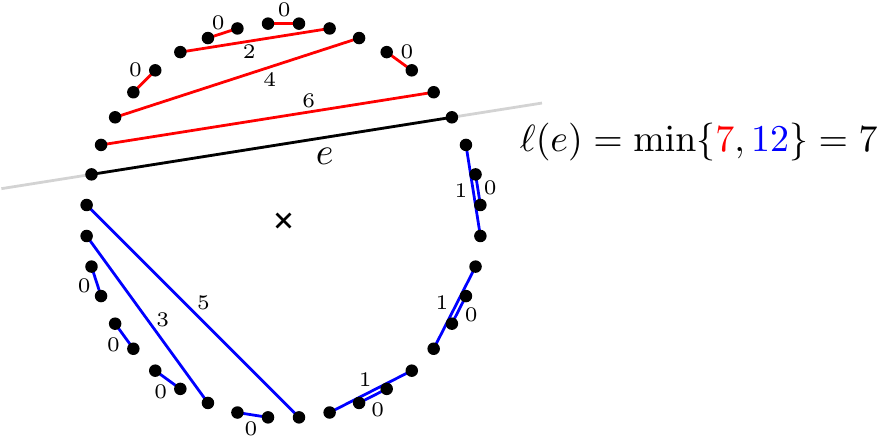}
}
\caption{Illustration of the definition of the length of an edge.
Each edge is labeled with its length.
}
\label{fig:edge}
\end{figure}

\begin{lemma}
\label{lem:centered}
A flip is centered if and only if the sum of the lengths of the four edges of the corresponding quadrilateral equals~$n-2$.
On the other hand, the flip is non-centered if and only if this sum is strictly less than~$n-2$.
\end{lemma}

We say that an edge $e\in M$ is \emph{visible} from the circle center, if the rays from the center to both edge endpoints do not cross any other matching edges.
If $n$ is odd, there may be an edge through the circle center, and then we decide visibility of the other edges by ignoring this edge, and declare the edge itself to \emph{not} be visible.

We also say that a point is \emph{hidden behind} a matching edge~$e$, if the point is not one of the endpoints of the edge and the ray from the center of the circle to the point crosses~$e$.
Similarly, we say that a matching edge~$f$ is \emph{hidden behind}~$e$, if both endpoints of~$f$ are hidden behind~$e$.
If $n$ is odd, there may be an edge through the circle center, and then no other points or edges are hidden behind this edge.

Treating the special case of an edge through the circle center in the way defined above simplifies and unifies some of the statements later (for example Theorem~\ref{thm:deg}).

\begin{lemma}
\label{lem:side}
Consider a matching $M\in\cM_n$ and a line $\rho$ between two antipodal points on the circle that are both endpoints of two visible matching edges~$e$ and~$f$.
If $n$ is odd, then $e$ and~$f$ lie on the same side of~$\rho$.
If $n$ is even, then $e$ and~$f$ lie on opposite sides of~$\rho$.
\end{lemma}

\begin{proof}
As $e$ and~$f$ are both visible, no other edge from~$M$ crosses the line~$\rho$.
Also note that the number of points on both sides of~$\rho$ is the same.
Consequently, if $e$ and $f$ lie on the same side of~$\rho$, then the number~$n$ of matching edges must be odd, and if they lie on opposite sides of~$\rho$, then the number~$n$ of matching edges must be even.
\end{proof}

\section{Vertex degrees}
\label{sec:deg}

In this section we characterize the degrees of vertices in the graph~$\cH_n$ by properties of the corresponding matchings (Theorem~\ref{thm:deg}), which allows us to determine the maximum and minimum degree of the graph~$\cH_n$, and to give characterizations and counts for the corresponding matchings (Theorem~\ref{thm:max-deg} and~\ref{thm:min-deg}, respectively).  

\begin{theorem}
\label{thm:deg}
Consider a matching $M\in\cM_n$.
If $n$ is odd, then the number of centered flips in~$M$ equals the number of visible edges.
If $n$ is even, then the number of centered flips in~$M$ is at most the number of visible edges in~$M$ and at least half this number, and both of these bounds are tight.
\end{theorem}

\begin{proof}
We first assume that $n\geq 3$ is odd.
The statement is obvious if $M$ has an edge of length~$\mu$ through the circle center, as in this case every visible edge can be flipped (only) together with this longest edge.
Now suppose that there is no edge of length~$\mu$ in~$M$.
Consider the set~$E$ of all visible edges, and let $F$ be the set of all unordered pairs of edges from~$E$ that are flippable together in a centered flip.
For any edge~$e\in E$, let $p$ and~$q$ be the endpoints of~$e$ such that the circle center is to the right of the ray from~$p$ to~$q$, and let $\tau(e)$ be the visible edge from~$M$ different from~$e$ that intersects the ray starting at~$p$ through the circle center.
Consider the mapping $\varphi:E\rightarrow F$ defined by $\varphi(e):=\{e,\tau(e)\}$; see Figure~\ref{fig:deg}.
First note that $e$ and $\tau(e)$ are flippable together in a centered flip, so $\varphi(e)$ is indeed a pair from~$F$.
We proceed to show that $\varphi$ is a bijection, which will prove the first part of theorem.
Indeed, for any pair $\{e,f\}\in F$ of visible edges that are flippable together in a centered flip, one can verify directly that $f=\tau(e)$ or $e=\tau(f)$ (or both), so the pair $\{e,f\}$ appears in the image of~$\varphi$, proving that $\varphi$ is surjective.
It remains to show that $\varphi$ is injective.
Suppose for the sake of contradiction that there are two distinct edges $e,f\in E$ with $\varphi(e)=\varphi(f)$, which means that $f=\tau(e)$ and $e=\tau(f)$.
This implies that there is a line~$\rho$ through the circle center and an endpoint of each of the two edges~$e$ and~$f$, such that $e$ and~$f$ lie on opposite sides of~$\rho$, contradicting Lemma~\ref{lem:side}.
Note that the assumption of $n$ being odd was only used in proving that $\varphi$ is injective.

\begin{figure}
\includegraphics{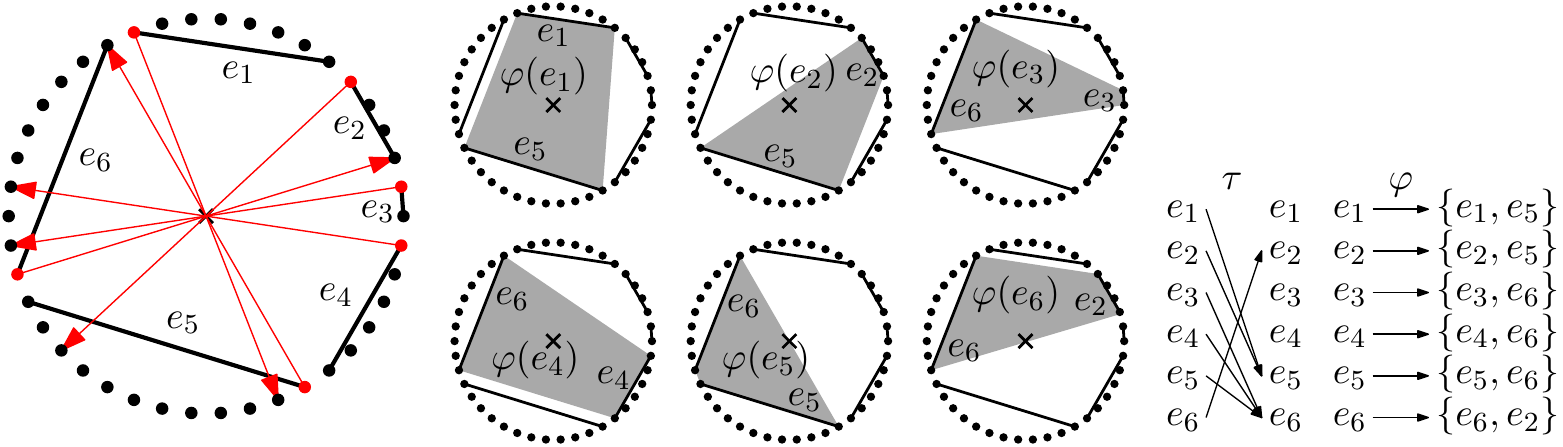}
\caption{
Illustration of the proof of Theorem~\ref{thm:deg} in the case when $n$ is odd.
The arrows indicate the rays that define the mapping~$\tau$.
The gray quadrilaterals show all possible centered flips in the given matching.
}
\label{fig:deg}
\end{figure}

We now assume that $n\geq 2$ is even.
Consider the mapping~$\varphi$ between visible edges~$E$ of~$M$ and flips~$F$ in~$M$ defined as before.
This mapping is surjective, as shown before, proving that the number of flips in~$M$ is at most the number of visible edges.
Even though the mapping~$\varphi$ may not be injective, there are at most two edges~$e$ and~$f$ mapping to the same pair $\{e,f\}\in F$ (if and only if $f=\tau(e)$ and $e=\tau(f)$).
This shows that the number of flips in~$M$ is at least half the number of visible edges.
Examples where these bounds are tight can be easily constructed; see Figure~\ref{fig:h6}.
\end{proof}

We now apply Theorem~\ref{thm:deg} to characterizing and counting vertices of maximum degree in the graph~$\cH_n$.

\begin{theorem}
\label{thm:max-deg}
For odd~$n\geq 3$, the graph~$\cH_n$ has maximum degree~$n$, and there are exactly two vertices of this degree, given by the two matchings that have only perimeter edges.

For even~$n\geq 2$, the graph~$\cH_n$ has maximum degree~$n/2$.
\end{theorem}

It seems to be considerably harder to count the number of vertices of maximum degree in~$\cH_n$ for even~$n$.
With computer help, we determined that these numbers are $2,10,54,274,1326,6218,28538$ for $n=2,4,6,8,10,12,14$, and they seem to defy a straightforward combinatorial interpretation.

\begin{proof}
The result for odd~$n\geq 3$ is an immediate consequence of the first part of Theorem~\ref{thm:deg}.

Now assume that $n\geq 2$ is even and consider a matching~$M\in\cM_n$.
Let $E$ be the set of visible edges of~$M$, and let $F$ be the set of unordered pairs of edges from~$E$ that are flippable together in a centered flip.
Also, define the mapping $\varphi:E\rightarrow F$ by $\varphi(e):=\{e,\tau(e)\}$ as in the proof of Theorem~\ref{thm:deg}.
We argued before that $\varphi$ is surjective, and that at most two edges $e,f\in E$ are mapped to the same pair $\{e,f\}\in F$.
We now consider the set $E'\seq E$ of edges for which the image under~$\varphi$ is unique.
Clearly, the total number of flips in~$M$ is
\begin{equation}
\label{eq:Fsize}
|F|=|E\setminus E'|/2+|E'|=(|E\setminus E'|+2|E'|)/2.
\end{equation}

Moreover, let $H$ be the set of hidden edges in~$M$, i.e., we have
\begin{equation}
\label{eq:msum}
n=|E\setminus E'|+|E'|+|H|.
\end{equation}

\begin{figure}
\includegraphics{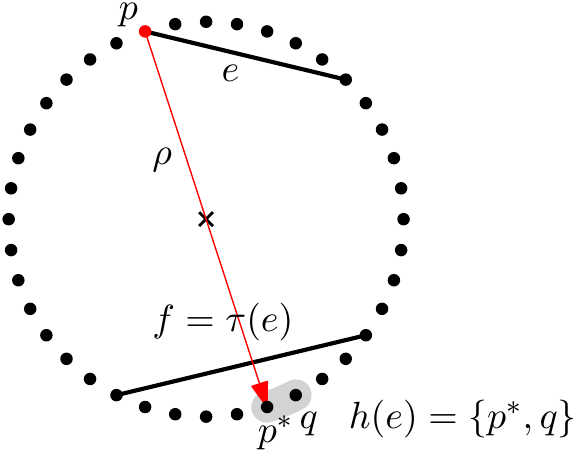}
\caption{
Illustration of the proof of Theorem~\ref{thm:max-deg} in the case when $n$ is even.
The arrow indicates the definition of the point~$p^*$.
The two points in the set $h(e)=\{p^*,q\}$ are shaded.
}
\label{fig:hidden}
\end{figure}

Now consider an edge $e\in E'$ and the pair $\{e,f\}\in F$ with $f=\tau(e)$ and $e\neq \tau(f)$, and let $\rho$ be the corresponding line starting at an endpoint~$p$ of~$e$ and intersecting~$f$, and let $p^*$ be the antipodal point to~$p$ on the circle; see Figure~\ref{fig:hidden}.
As a consequence of $e\neq \tau(f)$ and Lemma~\ref{lem:side}, the point~$p^*$ is hidden behind~$f$.
Let $q$ be the point next to~$p^*$ in counterclockwise direction, and define $h(e):=\{p^*,q\}$.
The point~$q$ must also be hidden behind~$f$.
If not, then it would have to be an endpoint of~$f$, and then flipping $f$ together with the edge ending at~$p^*$ in a \emph{non-centered} flip would produce a matching for which the line~$\rho$ violates Lemma~\ref{lem:side}.
Observe also that $h(e)$ and~$h(e')$ are disjoint sets for any two distinct edges $e,e'\in E'$, implying that
\begin{equation}
\label{eq:Hbound}
|H|\geq |E'|.
\end{equation}
For this argument it is irrelevant whether the two points in $h(e)$, $e\in E'$, are joined by an edge from~$H$ or not.

Plugging~\eqref{eq:Hbound} into~\eqref{eq:msum} shows that $|E\setminus E'|+2|E'|\leq n$, and using this estimate in~\eqref{eq:Fsize} proves that $|F|\leq n/2$, as desired.
\end{proof}

We conclude this section by characterizing and counting vertices of minimum degree in the graph~$\cH_n$.

\begin{theorem}
\label{thm:min-deg}
For odd~$n\geq 3$, the graph~$\cH_n$ has minimum degree~2, and there are exactly $n\cdot(C_{(n-3)/2})^2$ vertices of this degree, given by matchings that contain one edge of length~$\mu$ through the circle center, and two edges of length~$\mu-1$.

For even~$n\geq 2$, the graph~$\cH_n$ has minimum degree~1, and there are exactly $n\cdot(C_{(n-2)/2})^2$ vertices of this degree, given by matchings that contain exactly two edges of length~$\mu$.
\end{theorem}

\begin{proof}
First suppose that $n\geq 3$ is odd.
Clearly, any matching has at least two visible edges, so the first part of Theorem~\ref{thm:deg} shows that the minimum degree of~$\cH_n$ is~2.
Moreover, degree~2 is attained for exactly those matchings that have exactly two visible edges.
It remains to show that matchings with exactly one edge of length~$\mu$ through the circle center, and two edges of length~$\mu-1$, are the only ones that have exactly two visible edges.
This will prove the characterization of minimum degree vertices in the theorem.
The counting formula follows immediately, by observing that such a matching, apart from the three longest edges, consists of two independent matchings with $(n-3)/2$ edges each.
The $n$ possible rotations contribute the factor~$n$ in the counting formula.

Observe that if~$M$ does not contain an edge through the circle center, then it has at least three visible edges.
Indeed, any line between two antipodal points on the circle will touch at least two distinct edges~$e$ and~$f$ that are both visible and flippable together in a centered flip.
If these are the only visible edges, then there are no other points that are not endpoints of~$e$ or~$f$ or not hidden behind one of them.
As neither~$e$ nor~$f$ goes through the circle center, each of these two edges has length at most~$\mu-1$.
It follows that the sum of the four edge lengths of the quadrilateral involved in this flip is at most $2(\mu-1)=2((n-1)/2-1)=n-3$ (recall~\eqref{eq:mu}), contradicting Lemma~\ref{lem:centered}, which says that this sum is $n-2$ for a centered flip.

On the other hand, if~$M$ has an edge of length~$\mu$ through the circle center, then on each side of this edge at least one edge from~$M$ is visible, and equality occurs exactly if these two edges have length~$\mu-1$ each.

Now suppose that $n\geq 2$ is even.
Any matching has at least two visible edges, so the second part of Theorem~\ref{thm:deg} shows that the minimum degree of~$\cH_n$ is~1.
Moreover, degree~1 is attained exactly for those matchings that have exactly two visible edges.
It remains to show that matchings with exactly two edges of length~$\mu$ are the only ones that have exactly two visible edges.
This will prove the characterization of the minimum degree vertices and the counting formula.

Note that as $n$ is even, no matching edge can go through the circle center, otherwise there would be the same number of remaining edges on each of its two sides, which would make the total number of edges odd.
We show that if no edge goes through the circle center, and if the matching does not contain two edges of length~$\mu$, then the matching has at least three visible edges, making the degree of the corresponding vertex in~$\cH_n$ at least $\lceil 3/2\rceil=2$ by the second part of Theorem~\ref{thm:deg}. 
Similarly to before, by considering a line between two antipodal points on the circle, we obtain two visible and flippable edges~$e$ and~$f$.
The sum of the lengths of~$e$ and~$f$ is at most $2\mu-1$, by the assumption that the matching does not contain two edges of length~$\mu$.
However, these two edges would span a quadrilateral whose sum of edge lengths is at most $2\mu-1=2(n-2)/2-1=n-3$ (recall~\eqref{eq:mu}), contradicting Lemma~\ref{lem:centered}.

This completes the proof of the theorem.
\end{proof}

\section{Connectedness and diameter for odd~$n$}
\label{sec:odd}

In this section, we assume that the number~$n$ of matching edges is odd.
We prove that the graph~$\cH_n$ is connected in this case (Theorem~\ref{thm:conn}), and that its diameter is linear in~$n$ (Theorem~\ref{thm:diam}).

\begin{theorem}
\label{thm:conn}
For odd $n\geq 3$, the graph~$\cH_n$ is connected.
\end{theorem}

The proof of Theorem~\ref{thm:conn} is based on the following key lemma.
We consider two special matchings, namely those that have only perimeter edges, and we denote them by~$M_0$ and~$M_0'$; see Figure~\ref{fig:adidas}.

\begin{lemma}
\label{lem:4flips}
Consider a matching~$M\in\cM_n$ that has no edge through the circle center and that is different from~$M_0$ and~$M_0'$, i.e., $M$ has an edge of length strictly more than~0.
There is a sequence of at most~4 centered flips from~$M$ to another matching that has no edge through the circle center and that has at least one more visible edge than~$M$.
\end{lemma}

\begin{proof}
We fix a longest edge~$a$ in~$M$.
By the assumptions of the lemma, $a$ does not go through the circle center, and it must be visible in~$M$; see Figure~\ref{fig:4flips}~(a).

\begin{figure}
\makebox[0cm]{ 
\includegraphics{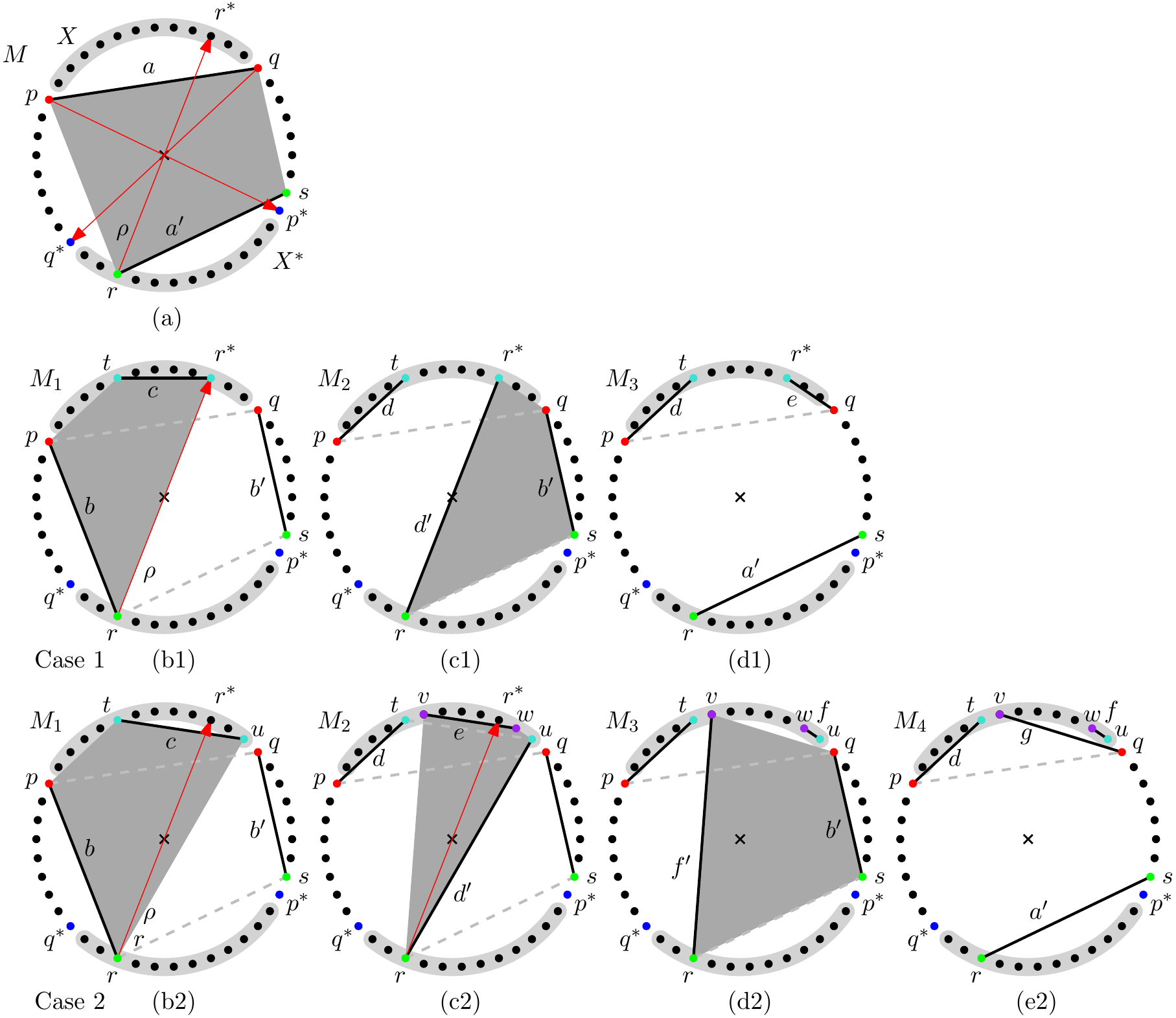}
}
\caption{Illustration of the proof of Lemma~\ref{lem:4flips}.
The dark-gray quadrilaterals indicate flips, and the light-gray circular arcs highlight the points sets~$X$ and~$X^*$.
If one of the edges~$a$ or~$a'$, which are both present in~$M$, is not present in one of the other matchings, then it is drawn as a dashed gray line in those matchings.
}
\label{fig:4flips}
\end{figure}

Let $p$ and~$q$ be the endpoints of the edge~$a$, such that the circle center is to the right of the ray from~$p$ to~$q$.
Moreover, let~$X$ be the set of points hidden behind the edge~$a$, and let $p^*$, $q^*$, and $X^*$ be the points or point sets, respectively, obtained by reflecting $p$, $q$, and~$X$ at the circle center.
We first argue that there is a matching edge $a'$ in~$M$ that is visible from the circle center and that has at least one endpoint in the set~$X^*$.
First of all, as $a$ is a longest edge, there is no edge in~$M$ such that $X\cup\{p,q\}$, $X^*\cup\{p^*,q^*\}$, $X^*\cup\{p^*\}$, or $X^*\cup\{q^*\}$ would be hidden behind it.
Also, the edge~$\{p^*,q^*\}$ is not in~$M$ by Lemma~\ref{lem:side} (this edge would hide precisely the points in~$X^*$).

We conclude that there is an edge~$a'$ in~$M$ that is visible from the circle center and that has one endpoint~$r$ in the set~$X^*$.
Let $r^*$ be the antipodal point to~$r$ on the circle, and note that $r^*\in X$ by the fact that $r\in X^*$.
We now consider the line~$\rho$ through~$r$ and~$r^*$.
We assume w.l.o.g.\ that the other endpoint~$s$ of the edge~$a'$ different from~$r$ is to the right of the ray from~$r$ to~$r^*$ (otherwise mirror the configuration).
Clearly, the edge~$a'$ is flippable together with~$a$.
Performing this centered flip yields a matching~$M_1$ with two new edges that we call~$b$ and~$b'$, where $b$ has endpoints $p$ and~$r$, and $b'$ has endpoints~$q$ and~$s$.
We now distinguish two cases: Either~$\rho$ does not cross any edges of the matching~$M_1$, or $\rho$ crosses some matching edge.

Case 1: In~$M_1$, the line~$\rho$ does not cross any matching edges; see Figure~\ref{fig:4flips}~(b1).
Let $c$ be the edge ending at~$r^*\in X$.
Note that the other endpoint~$t$ of~$c$ must also be in~$X$, as the edge~$c$ was hidden behind~$a$ in~$M$.
Furthermore, $b$ and~$c$ lie on the same side of the line~$\rho$ by Lemma~\ref{lem:side}.
We can thus flip $b$ together with~$c$, yielding a matching~$M_2$ with two new edges that we call~$d$ and~$d'$, where $d'$ is the edge through the circle center with endpoints~$r$ and~$r^*$ and~$d$ has~$p$ and~$t$ as endpoints; see Figure~\ref{fig:4flips}~(c1).
Now, we flip $d'$ together with $b'$, yielding a matching~$M_3$ which again contains the edge~$a'$, plus a new edge that we call~$e$, which has~$q$ and~$r^*$ as endpoints; see Figure~\ref{fig:4flips}~(d1).
Observe that $M_3$ differs from~$M$ only in the removal of the edges~$a$ and~$c$, and the addition of~$d$ and~$e$ that are both shorter than~$a$.
As the edge~$c$ is hidden behind~$a$ in~$M$ and therefore not visible, and the edges $d$ and~$e$ are visible in~$M_3$, the lemma is proved in this case.

Case 2: In~$M_1$, the line~$\rho$ crosses some matching edge~$c$; see Figure~\ref{fig:4flips}~(b2).
We let $t$ and $u$ denote the endpoints of~$c$, where the circle center is to the right of the ray from~$t$ to~$u$.
Both $t$ and~$u$ are in~$X$, as the edge~$c$ was hidden behind the edge~$a$ in~$M$.
We flip $b$ together with~$c$, yielding a matching~$M_2$ with two new edges that we call~$d$ and~$d'$, where $d$ has $p$ and~$t$ as endpoints, and $d'$ has $r$ and~$u$ as endpoints; see Figure~\ref{fig:4flips}~(c2).
As the point~$r^*$ lay behind the edge~$c$ in~$M_1$, there must be another edge~$e$ behind~$c$ in~$M_1$ which touches or crosses the line~$\rho$ and is now visible in~$M_2$.
We let $v$ and $w$ denote the endpoints of this edge, where the circle center is to the right of the ray from~$v$ to~$w$.
We flip the edge~$d'$ together with~$e$, yielding a matching~$M_3$ with two new edges that we call~$f$ and~$f'$, where $f$ has $u$ and~$w$ as endpoints, and $f'$ has $r$ and~$v$ as endpoints; see Figure~\ref{fig:4flips}~(d2).
We finally flip the edge~$f'$ together with~$b'$, yielding a matching~$M_4$ which again contains the edge~$a'$, plus a new edge that we call~$g$, which has~$q$ and~$v$ as endpoints; see Figure~\ref{fig:4flips}~(e2).
Observe that $M_4$ differs from~$M$ only in the removal of the edges~$a$, $c$, and $e$, and the addition of~$d$, $f$ and~$g$ that are all shorter than~$a$.
As the edges~$c$ and~$e$ are hidden behind~$a$ in~$M$ and therefore not visible, and the edges~$d$ and~$g$ are visible in~$M_4$, the lemma is proved in this case.
\end{proof}

With Lemma~\ref{lem:4flips} in hand, the proof of Theorem~\ref{thm:conn} is straightforward.

\begin{proof}[Proof of Theorem~\ref{thm:conn}]
\begin{figure}
\includegraphics{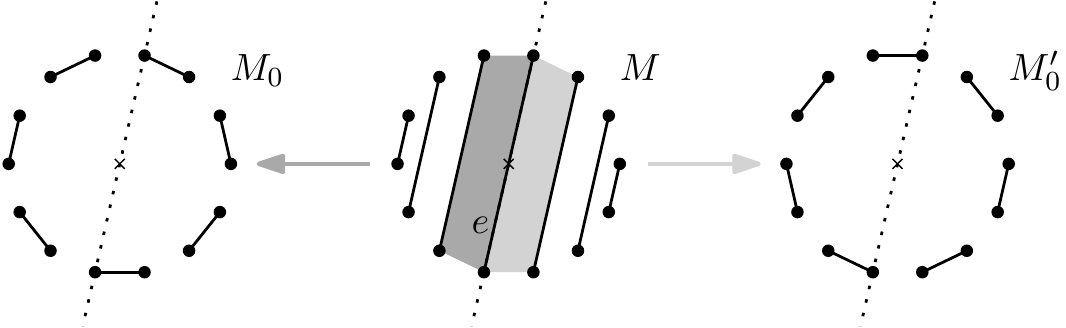}
\caption{Illustration of the proof of Theorem~\ref{thm:conn}.
The two gray quadrilaterals indicate the only two centered flips in~$M$.}
\label{fig:adidas}
\end{figure}

Let $M\in\cM_n$ be an arbitrary matching.
If $M$ has an edge through the circle center, there is a centered flip involving this edge to a matching without an edge through the center.
If $M$ has no edge through the circle center, then repeatedly applying Lemma~\ref{lem:4flips} shows that in the graph~$\cH_n$ there is a path from~$M$ to either~$M_0$ or~$M_0'$, as these are the only two matchings that have the maximum number of visible edges.
To prove the theorem, it suffices to show that there is also a path between~$M_0$ and~$M_0'$ in~$\cH_n$.
To see this, consider a matching~$M$ that consists of $n$ parallel edges; see Figure~\ref{fig:adidas}.
We have argued before that there is a path between~$M$ and either~$M_0$ or~$M_0'$.
W.l.o.g.\ we assume that this path reaches~$M_0$.
The first centered flip on this path must involve the unique edge~$e$ of length~$\mu$ of~$M$, and one of the two edges of length~$\mu-1$ next to it.
Consider the sequence of quadrilaterals corresponding to the flip sequence on this path from~$M$ to~$M_0$, and consider the quadrilaterals obtained by mirroring along the line through~$e$.
This mirrored flip sequence will lead from~$M$ to the matching obtained from~$M_0$ by mirroring along the line through~$e$, which is precisely the matching~$M_0'$.
This proves that $M_0$ and~$M_0'$ are connected in~$\cH_n$, and this completes the proof of the theorem.
\end{proof}

\begin{figure}
\includegraphics{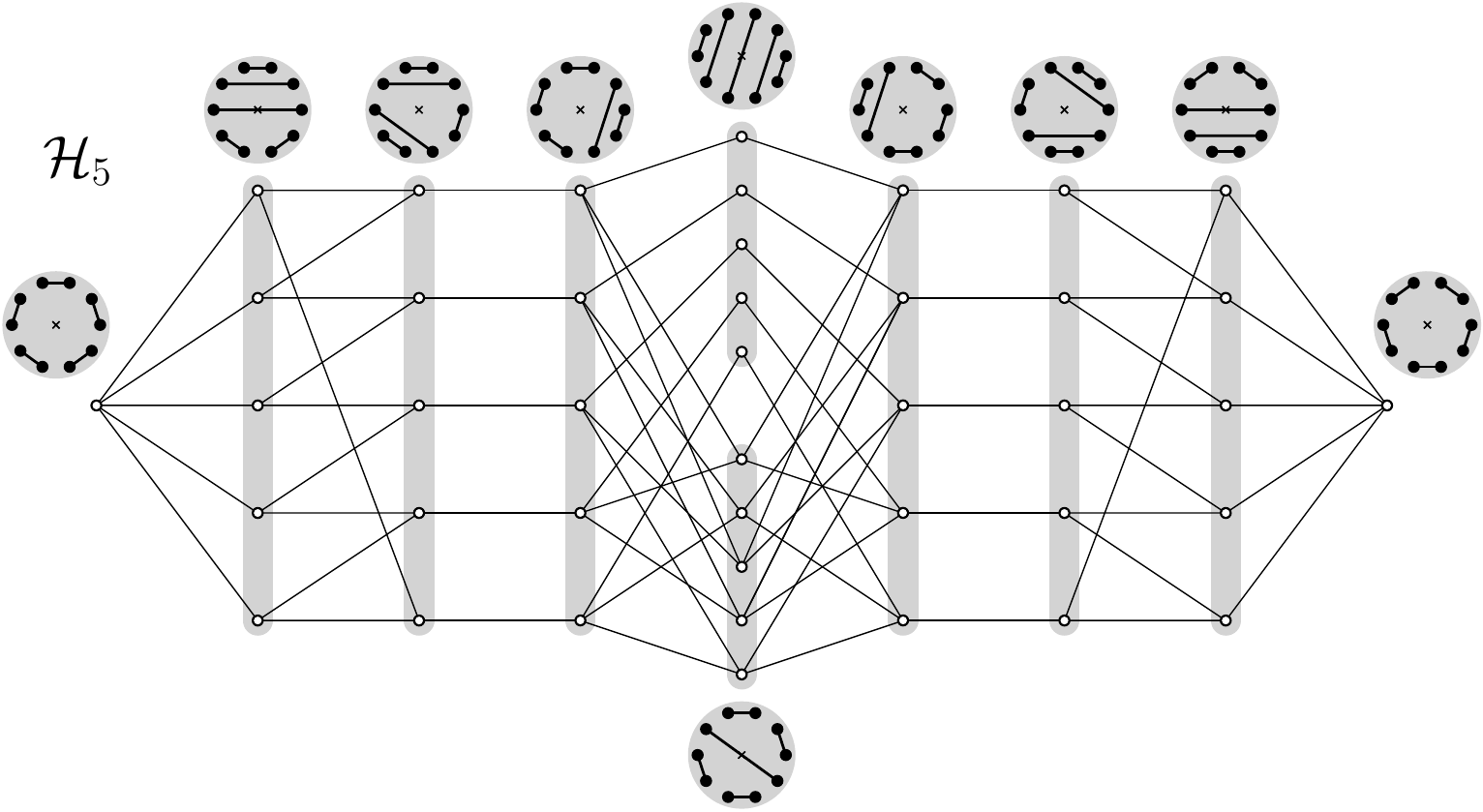}
\caption{The graph $\cH_5$, drawn in a simplified way, where for every matching, we only show one representative of the equivalence class under rotation by~$2\pi/5$.}
\label{fig:h5}
\end{figure}

Recall that the \emph{diameter} of a graph is the maximum length of all shortest paths between any two vertices of the graph.
In the flip graph~$\cH_n$, the diameter measures how many centered flips are needed in the worst case to transform two matchings into each other.
With computer help, we determined the diameter of~$\cH_n$ for $n=3,5,7,9,11$ to be~$2,8,14,20,26$, which equals~$3n-7$ for those values of~$n$.
In all those cases, this distance was attained for the two matchings that have only perimeter edges (differing by a rotation of~$\pi/n$).
These are the extreme vertices on the left and right in Figure~\ref{fig:h5} (cf.\ also the left hand side of Figure~\ref{fig:g34}).
We conjecture that this is the correct value for all odd~$n$.
As a first step towards this conjecture, we can prove the following linear bounds.

\begin{theorem}
\label{thm:diam}
For odd $n\geq 3$, the diameter of~$\cH_n$ is at least $n-1$ and at most $11n-29$.
\end{theorem}

\begin{proof}
Hernando, Hurtado, and Noy~\cite{MR1939072} showed that the diameter of~$\cG_n$ is exactly~$n-1$, and as $\cH_n$ is a spanning subgraph of~$\cG_n$, its diameter is at least~$n-1$.

It remains to prove the upper bound in the theorem.
As before, we let $M_0$ and~$M_0'$ denote the two matchings that have only perimeter edges.
We first argue that the distance between any matching~$M\in\cM_n$ and either~$M_0$ or~$M_0'$ is at most~$4n-11$.
Indeed, if $M$ has no edge through the circle center, then it has at least~3 visible edges by Lemma~\ref{lem:side}.
As a consequence of Lemma~\ref{lem:4flips}, we can reach~$M_0$ or~$M_0'$, which have $n$ visible edges each, from~$M$ with at most $4(n-3)=4n-12$ centered flips.
On the other hand, if~$M$ has an edge through the circle center, then a single centered flip leads from~$M$ to one of its neighbors that does not have an edge through the center, establishing the bound~$4n-11$.

In the remainder of the proof we show that the distance between~$M_0$ and~$M_0'$ in~$\cH_n$ is at most~$3n-7$.
With these two bounds, we can then bound the distance in~$\cH_n$ between any two matchings~$M,M'\in\cM_n$ as follows:
We know that from both~$M$ and~$M'$ we can reach either~$M_0$ or~$M_0'$ with at most~$4n-11$ centered flips each.
If both of these flip sequences reach the same matching from~$\{M_0,M_0'\}$, we have found a path in~$\cH_n$ of length at most~$2(4n-11)$ between~$M$ and~$M'$.
Otherwise we can connect~$M_0$ and~$M_0'$ with a path of length~$3n-7$, yielding a path of length at most~$2(4n-11)+(3n-7)=11n-29$ between~$M$ and~$M'$, proving the upper bound in the theorem.

\begin{figure}
\includegraphics{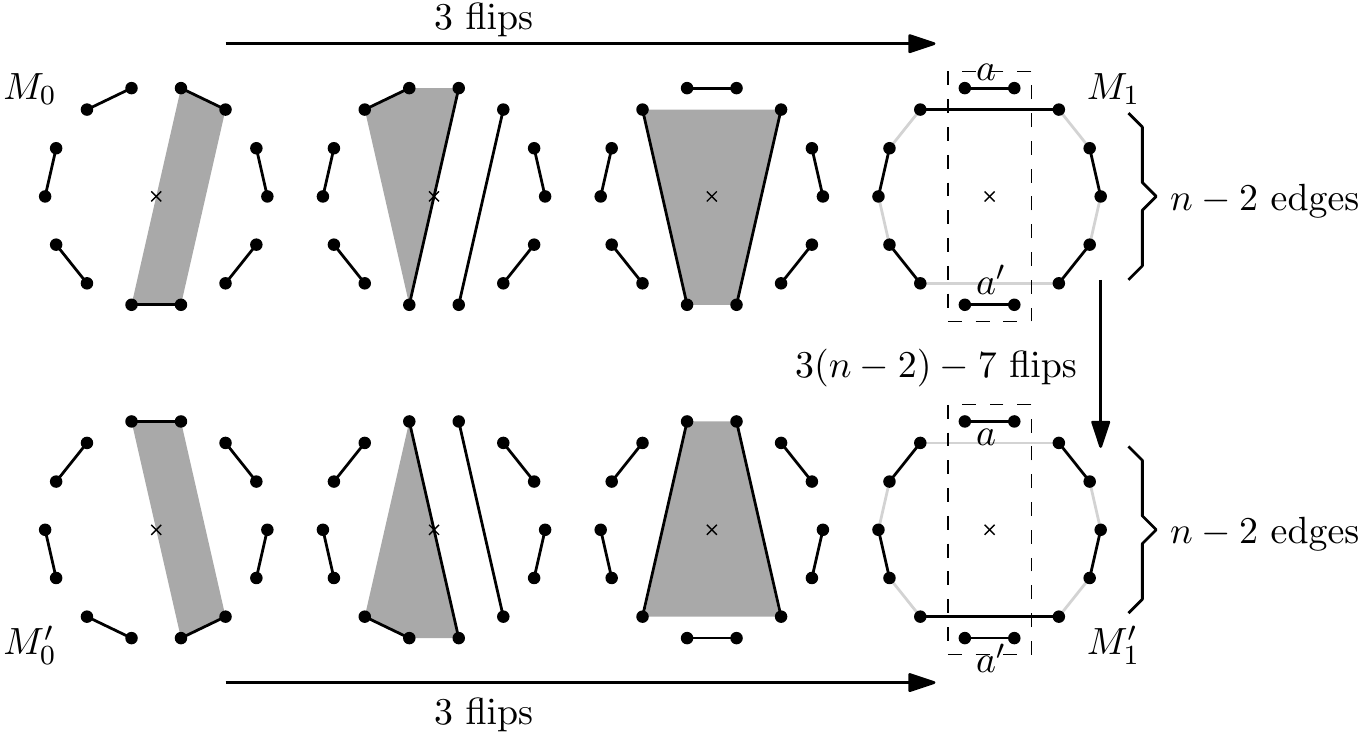}
\caption{Illustration of the inductive proof that the distance between~$M_0$ and~$M_0'$ in~$\cH_n$ is at most~$3n-7$.
The two antipodal edges~$a$ and~$a'$ that are the same in~$M_1$ and~$M_1'$ are framed by a dashed box.
The edges in the symmetric difference of~$M_1$ and~$M_1'$ are drawn light-gray in these two matchings.
The gray quadrilaterals indicate flips between~$M_0$ and~$M_1$, and between~$M_0'$ and~$M_1'$, respectively.
}
\label{fig:stack}
\end{figure}

We prove the claim that the distance between~$M_0$ and~$M_0'$ is at most~$3n-7$ by induction on all odd values of~$n\geq 3$; see Figure~\ref{fig:stack}.
For $n=3$ the distance between $M_0$ and~$M_0'$ is $3n-7=2$, as can be verified from the left hand side of Figure~\ref{fig:g34}.
For the induction step, suppose that $n\geq 5$ is odd and that the claim holds for~$n-2$.
Consider the flip sequence shown at the top part of Figure~\ref{fig:stack}, consisting of 3 centered flips, leading from the matching~$M_0$ to a matching~$M_1$ that contains $n-1$ perimeter edges and one edge of length~1.
Consider the two edges~$a$ and~$a'$ in~$M_1$ that lie antipodally on the circle, where the edge~$a$ is hidden behind the unique length-1 edge.
We can ignore the edges~$a$ and~$a'$ from the configuration, and obtain a matching with $n-2$ edges.
As the ignored edges are antipodal on the circle, every centered flip operating on the remaining $n-2$ edges in~$\cH_{n-2}$ is also a centered flip in~$\cH_n$.
Also observe that ignoring those two edges from~$M_1$ leaves us with a matching in~$\cH_{n-2}$ that has only perimeter edges.
Consequently, by induction we have a flip sequence of length~$3(n-2)-7$ from~$M_1$ to a matching~$M_1'$ that has $n-1$ perimeter edges and one edge of length~1, that still contains the edges~$a$ and~$a'$, but now the edge~$a'$ is hidden behind the unique length-1 edge.
By symmetry, we can reach $M_0'$ from $M_1'$ with at most 3 centered flips.
Overall, the length of the flip sequence from~$M_0$ to~$M_0'$ obtained in this way is $3(n-2)-7+2\cdot 3=3n-7$.
This completes the inductive proof and thus the proof of the theorem.
\end{proof}

\section{Component structure for even~$n$}
\label{sec:even}

In this section, we assume that the number~$n$ of matching edges is even.
It was proved in~\cite{MR4046775} that in this case the graph~$\cH_n$ has at least $n-1$ components.
We improve upon this considerably, by showing that $\cH_n$ has \emph{exponentially} many components, and we also provide a fine-grained picture of the component structure of the graph~$\cH_n$ (Theorem~\ref{thm:struct} and Corollary~\ref{cor:comp}).
We also prove explicit formulas for the number of matchings with certain weights, a parameter that is closely related to the component sizes of the graph~$\cH_n$, proving a conjecture raised in~\cite{MR4046775} (Theorem~\ref{thm:narayana} and Corollary~\ref{cor:size}).

\subsection{Centrally symmetric matchings}

A matching~$M\in\cM_n$ is said to be \emph{centrally symmetric}, if it is invariant under point reflection at the circle center.
We let $\cS_n\seq\cM_n$ denote the set of all centrally symmetric matchings.
For any edge~$e$ in a matching~$M\in\cS_n$, we let $\sigma(e)$ denote the edge that is obtained from~$e$ by point reflection at the circle center; see Figure~\ref{fig:length}.

\begin{theorem}
\label{thm:struct}
For even~$n\geq 2$, there are $\binom{n}{n/2}$ centrally symmetric matchings, and all those matchings form components in~$\cH_n$ that are trees.
There are $C_{n/2}$ such components, and each of them contains exactly $n/2+1$ matchings.
All matchings that are not centrally symmetric form components that are not trees.
\end{theorem}

The properties of the graph~$\cH_n$ stated in Theorem~\ref{thm:struct} can be seen nicely in Figure~\ref{fig:h6} for the case $n=6$.

The proof of Theorem~\ref{thm:struct} is split into the following four lemmas.
Specifically, we will count point-symmetric matchings (Lemma~\ref{lem:point}), show that they form tree components in~$\cH_n$ (Lemma~\ref{lem:tree}), determine the size of those components (Lemma~\ref{lem:tree-size}), and show that no other matching lies in a tree component (Lemma~\ref{lem:not-tree}).
For the remainder of this paper, we give the points on the unit circle a fixed labeling by $1,2,\ldots,2n$ in clockwise direction.

\begin{lemma}
\label{lem:point}
For even $n\geq 2$, the number of centrally symmetric matchings is $\binom{n}{n/2}$.
\end{lemma}

\begin{figure}
\includegraphics{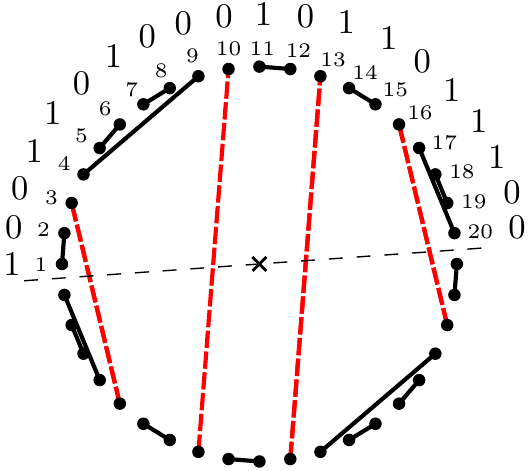}
\caption{
Bijection used in the proof of Lemma~\ref{lem:point}.
Matching edges from one of the first $n$ points to one of the last $n$ points are dashed, while the matching edges between points of the same group are drawn solid.
}
\label{fig:point}
\end{figure}

\begin{proof}
To prove the counting formula, we establish a bijection between~$\cS_n$ and binary strings of length~$n$ with exactly $n/2$ many~0s and~1s; see Figure~\ref{fig:point}.
For this we use the labels $1,2,\ldots,2n$ of the points.

Given a matching $M\in\cS_n$, we scan through the first half of the points, i.e., from point~1 to point~$n$.
For each of the points that we encounter, we record a 0-bit or a 1-bit as follows:
For every edge that has both endpoints among the points $1,\ldots,n$, we record a~1 when we encounter its first endpoint, and a~0 when we encounter its second endpoint.
For every edge~$e$ that has only one endpoint among the points $1,\ldots,n$, the edge $\sigma(e)$ has also exactly one endpoint among the points $1,\ldots,n$, and we record a~0 when we encounter the first of these two points, and a~1 when we encounter the second of these two points.
This procedure clearly yields a binary string of length~$n$ with exactly $n/2$ many~0s and~1s.

Given a binary string of length~$n$ with exactly $n/2$ many~0s and~1s, we append two copies of the string, yielding a string of length~$2n$, and we label the points $1,\ldots,2n$ along the circle with the bits of this string.
We then repeatedly match a point labeled with a~1 with a point labeled with a~0, subject to the constraint that all points hidden behind this matching edge are already matched.
As there are $n$ points labeled~0 and $n$ points labeled~1, the resulting matching is perfect.
Moreover, by construction every pair of antipodal points is labeled with the same bit, and consequently every matching edge~$e$ is present together with its partner~$\sigma(e)$, i.e., we indeed obtain a matching from~$\cS_n$.

One can verify directly that the mappings described before are inverse to each other.
\end{proof}

One last remark about Lemma~\ref{lem:point}: It can easily be shown that for \emph{odd} $n\geq 3$, the number of centrally symmetric matchings is $n\cdot C_{(n-1)/2}$, but this is not relevant here.

\begin{lemma}
\label{lem:tree}
For even $n\geq 2$, any centrally symmetric matching lies in a component of~$\cH_n$ that contains only centrally symmetric matchings, and this component is a tree.
\end{lemma}

\begin{proof}
First of all, any centered flip in a centrally symmetric matching produces another centrally symmetric matching, so centrally symmetric matchings lie in their own components.
Moreover, in a centrally symmetric matching, every edge~$e$ can only be flipped together with its centrally symmetric partner~$\sigma(e)$.

Suppose for the sake of contradiction that there was a cycle in~$\cH_n$ consisting of centrally symmetric matchings, and consider three consecutive (and therefore distinct) matchings $M_1,M_2,M_3\in\cS_n$ on that cycle; see Figure~\ref{fig:length}.
Let $e,\sigma(e)$ be the two edges that entered in the flip from~$M_1$ to~$M_2$, and let $f,\sigma(f)$ be the two edges that entered in the flip from~$M_2$ to~$M_3$.
Clearly, $e$ and~$\sigma(e)$ must be hidden behind~$f$ and~$\sigma(f)$, and therefore we have $\ell(f)>\ell(e)$.
This length increase holds for any triple of consecutive matchings on the cycle, which is a contradiction.

\begin{figure}
\includegraphics{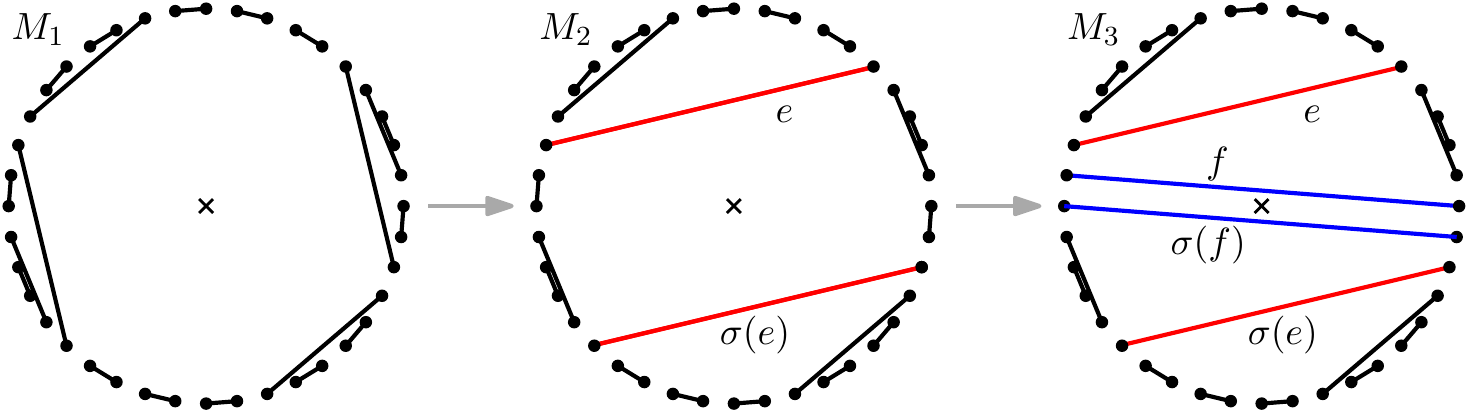}
\caption{
Illustration of the proof of Lemma~\ref{lem:tree}.
The edges $e$ and~$\sigma(e)$ that enter in the flip from~$M_1$ to~$M_2$, and the edges $f$ and~$\sigma(f)$ that enter in the flip from~$M_2$ to~$M_3$, are highlighted.
}
\label{fig:length}
\end{figure}
\end{proof}

\begin{lemma}
\label{lem:tree-size}
For even $n\geq 2$, any centrally symmetric matching lies in a component of~$\cH_n$ that contains exactly $n/2+1$ matchings.
\end{lemma} 

\begin{proof}
Let $M\in\cS_n$ be a centrally symmetric matching, and let $T$ be the component of~$\cH_n$ containing~$M$.
By Lemma~\ref{lem:tree} the component~$T$ is a tree containing only centrally symmetric matchings.

The matching $M$ consists of $n/2$ unordered pairs of edges $\{e,\sigma(e)\}$.
Any such pair of edges $\{e,\sigma(e)\}$ can flipped somewhere in~$T$.
Indeed, if $e$ and~$\sigma(e)$ are visible in~$M$, then we can flip them with a centered flip to reach a neighbor of~$M$ in~$T$.
On the other hand, if $e$ and~$\sigma(e)$ are not visible in~$M$, then they are hidden behind some other pairs of edges, and we can remove those by centered flips one after the other, starting from the innermost pair, until $e$ and~$\sigma(e)$ become visible and hence flippable together.
As every flip corresponds to an edge of~$T$, it follows that~$T$ has at least $n/2$ edges.

We now show that~$T$ has at most $n/2$ edges.
For this consider~$M$ and one of its neighbors~$M'$ on~$T$, and let $\{f,\sigma(f)\}$ be the edges that enter in the flip from~$M$ to~$M'$, and let $\{e,\sigma(e)\}$ be the edges that enter in the flip from~$M'$ to~$M$, i.e., the four edges $e,\sigma(e),f,\sigma(f)$ form the rectangle corresponding to this flip.
Consider a path in~$T$ that starts at~$M$ and moves to~$M'$ with its first flip.
Every matching encountered after~$M'$ on that path contains a pair of matching edges that hide~$f$ and~$\sigma(f)$, and so this flip involving~$f$ and~$\sigma(f)$ cannot occur again in this part of the tree~$T$.
Now consider a path in~$T$ that starts at~$M'$ and moves to~$M$ with its first flip.
Every matching encountered after~$M$ on that path contains a pair of matching edges that hide~$e$ and~$\sigma(e)$, and so this flip involving~$e$ and~$\sigma(e)$ cannot occur again in this part of the tree~$T$, either.
We conclude that~$T$ has at most $n/2$ edges.

Combining these observations, we conclude that~$T$ has exactly $n/2$ edges, and hence exactly $n/2+1$ vertices.
\end{proof}

\begin{lemma}
\label{lem:not-tree}
For even $n\geq 2$, any matching that is not centrally symmetric lies in a component of~$\cH_n$ that is not a tree.
\end{lemma}

For the proof of Lemma~\ref{lem:not-tree} we will need the following definitions:
We say that an edge~$e$ in a matching~$M$ is \emph{non-symmetric}, if the edge $\sigma(e)$ is not in~$M$.
We also say that a matching~$M\in\cM_n\setminus\cS_n$ is \emph{nice}, if one of its non-symmetric edges is visible.
The strategy for the proof is to show that every matching is connected to a nice matching in~$\cH_n$, and that every nice matching lies on a cycle in~$\cH_n$.

\begin{proof}
We claim that every nice matching~$M$ has two distinct nice matchings as neighbors in the graph~$\cH_n$.
As the number of matchings is finite, this implies that a nice matching must lie in a cycle of~$\cH_n$, i.e., in a component that is not a tree.
The lemma follows from this, by observing that from every matching~$M\in\cM_n\setminus\cS_n$, we can reach a nice matching by centered flips.

\begin{figure}
\includegraphics{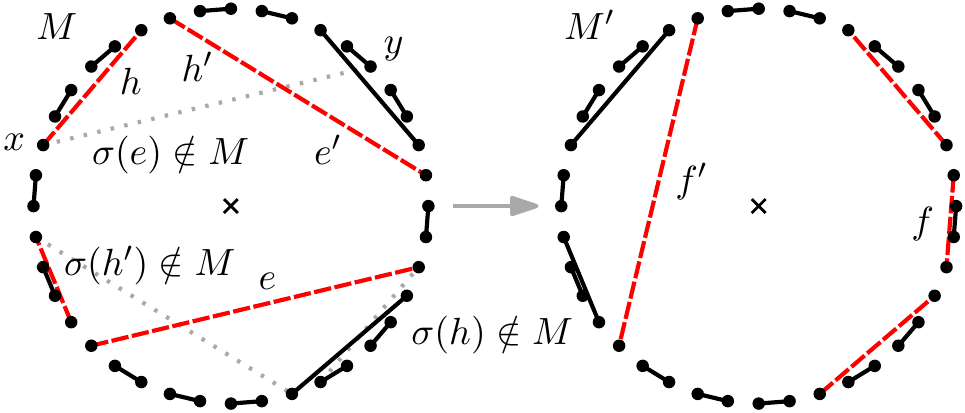}
\caption{
Notations used in the proof Lemma~\ref{lem:not-tree}.
Visible non-symmetric matching edges are dashed.
The dotted lines indicate the images of the edges~$e$, $h$ and~$h'$ under point reflection at the circle center, and these images are \emph{not} edges of the matching~$M$.
}
\label{fig:cycle}
\end{figure}

So let $M\in\cM_n\setminus\cS_n$ be a nice matching, and let $e$ be any visible non-symmetric edge of maximum length in~$M$; see Figure~\ref{fig:cycle}.
Consider an edge~$e'$ that forms a centered flip together with~$e$, and let~$M'$ be the matching obtained by performing this flip.
Moreover, let $f$ and $f'$ be the edges that appear in~$M'$ with this flip.
Note that $f\neq\sigma(f')$, as $e$ is non-symmetric.
Moreover, if $\sigma(f)\in M'$, then $\sigma(f)$ must be hidden behind~$f'$, implying that $\ell(f)<\ell(f')$.
Symmetrically, if $\sigma(f')\in M'$, then $\sigma(f')$ must be hidden behind~$f$, implying that $\ell(f')<\ell(f)$.
It follows that at least one of $f$ or $f'$ is non-symmetric and visible in~$M'$.

We now argue that $M$ contains at least three visible non-symmetric edges, so there are at least two nice neighbors of~$M$ in~$\cH_n$.
Let $x$ and~$y$ be the endpoints of the edge~$\sigma(e)$ (which is \emph{not} present in~$M$).
Observe that there is no edge~$g$ in~$M$ that hides $x$ and~$y$, or that hides~$x$ and has~$y$ as one of its endpoints, or that hides~$y$ and has~$x$ as one of its endpoints.
Clearly, $\sigma(g)\notin M$, as the edge~$e$ is visible.
Therefore, $g$ would have to be non-symmetric, but this contradicts our choice of~$e$ as the visible non-symmetric edge of maximum length.
It follows that there are two visible edges $h,h'$ in~$M$ that intersect the line between~$x$ and~$y$.
For both of them, we have $\sigma(h)\notin M$ and $\sigma(h')\notin M$, as the edges $\sigma(h)$ and $\sigma(h')$ would intersect the edge~$e\in M$.
Consequently, $h$ and~$h'$ are both visible and non-symmetric, so $M$ has at least three visible non-symmetric edges, as claimed.
\end{proof}

With these lemmas in hand, we are now ready to prove Theorem~\ref{thm:struct}.

\begin{proof}[Proof of Theorem~\ref{thm:struct}]
Combine Lemmas~\ref{lem:point}, \ref{lem:tree}, \ref{lem:tree-size}, and~\ref{lem:not-tree}, and observe that $\frac{1}{n/2+1}\binom{n}{n/2}=C_{n/2}$.
\end{proof}

\subsection{Weights of matchings}

For our further investigations, we assign an integer weight to each matching, again using the labels $1,\ldots,2n$ of the points.
Consider a matching~$M\in\cM_n$ and one of its edges~$e\in M$, and let~$i$ and~$j$ be the endpoints of~$e$ so that the circle center lies to the right of the ray from~$i$ to~$j$.
We define the \emph{sign of the edge~$e$} as
\begin{equation*}
\sgn(e):=\begin{cases}
          +1 & \text{if $i$ is odd}, \\
          -1 & \text{if $i$ is even}.
         \end{cases}
\end{equation*}
We call an edge~$e$ \emph{positive} if $\sgn(e)=+1$, and we call it \emph{negative} if $\sgn(e)=-1$. 
Moreover, we define the \emph{weight} of the matching~$M$ as
\begin{equation*}
w(M):=\sum_{e \in M} \sgn(e) \cdot \ell(e).
\end{equation*}
These definitions are illustrated in Figure~\ref{fig:weight}.
Note that rotating a matching by~$\pi/n$, either clockwise or counterclockwise, changes the weight by a factor of~$-1$.

\begin{figure}
\includegraphics{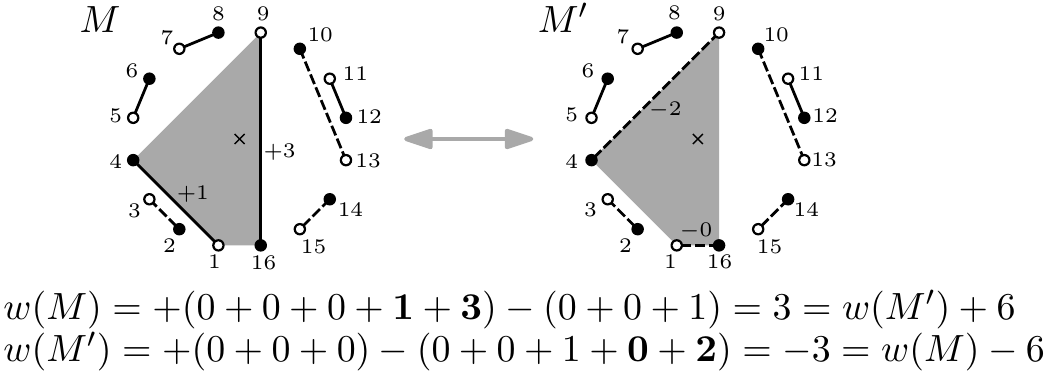}
\caption{Illustration of the weight of two matchings with~$n=8$ edges.
Odd points are drawn as white bullets, even points as black bullets.
Positive edges are drawn solid, negative edges are drawn dashed.
Note that the centered flip changes the weight by $\pm (n-2)=\pm 6$.
}
\label{fig:weight}
\end{figure}

Observe also that a quadrilateral corresponding to a centered flip has two positive edges and two negative edges, and the pairs of edges with the same sign are opposite to each other on the quadrilateral.
Combining this observation with Lemma~\ref{lem:centered} shows that any centered flip changes the weight of a matching by~$\pm(n-2)$.
Moreover, it was shown in~\cite{MR4046775} that all possible weights are in a particular integer range.
These results are summarized in the following lemma, which is illustrated in Figure~\ref{fig:weight}.

\begin{lemma}[Lemmas~11+12 in \cite{MR4046775}]
\label{lem:weight}
Let $n\geq 2$ be even.
Applying a centered flip to any matching from~$\cM_n$ changes its weight by~$-(n-2)$ if the two negative edges appear in this flip, or by~$+(n-2)$ if the two positive edges appear in this flip, and flips of these two kinds must alternate along any sequence of centered flips.
Moreover, for any matching~$M\in\cM_n$ we have
\begin{equation*}
w(M) \in [-(n-2), n-2] := \big\{-(n-2), -(n-2) +1, \dots, n-3, n-2 \big\},
\end{equation*}
and each of these weight values is attained for some matching in~$\cM_n$.
\end{lemma}

Our next result is an immediate consequence of this lemma.

\begin{corollary}
\label{cor:comp}
For even $n\geq 4$, the graph $\cH_n$ has at least $C_{n/2}+n-3$ components.
\end{corollary}

\begin{proof}
By Theorem~\ref{thm:struct}, the graph $\cH_n$ has exactly~$C_{n/2}$ components that contain all centrally symmetric matchings.
Moreover, for $c=1,\ldots,n-3$, we can easily construct a matching that is not centrally symmetric and has weight~$c$.
Indeed, for $c=1,\ldots,\mu$, we take a matching that has a single edge of length~$c$, and all other edges are perimeter edges (length~0).
For $c=\mu+1,\ldots,2\mu-1=n-3$, we take a matching with a single edge of length~$\mu$, another edge of length~$c$, and all other edges are perimeter edges.
By Lemma~\ref{lem:weight}, these $n-3$ matchings all lie in distinct components of~$\cH_n$, and they must be different from the~$C_{n/2}$ components containing the centrally symmetric matchings.
This implies the claimed lower bound.
\end{proof}

Motivated by Lemma~\ref{lem:weight}, we partition the set of all matchings~$\cM_n$ according to their weights.
Specifically, for any non-zero integer $c\in[-(n-2),(n-2)]$, we let $\cW_{n,c}$ be the set of all matchings from~$\cM_n$ with weight exactly~$c$.
For the special case $c=0$ we define $\cW_{n,0}:=\{M_0\}$ and $\cW_{n,0}^-:=\{M_0^-\}$, where $M_0$ is the matching that has only perimeter edges and all of them positive, and $M_0^-$ is the matching that has only perimeter edges and all of them negative.
Clearly, we have $w(M_0)=w(M_0^-)=0$.
Moreover, for $c=0,1,\dots,n-2$ we define
\begin{equation}
\label{eq:Mmc}
\cM_{n,c}:=\begin{cases} \cW_{n,c}\cup \cW_{n,c-(n-2)} & \text{if } c\leq n-3, \\
                         \cW_{n,n-2}\cup \cW_{n,0}^- & \text{if } c=n-2,
                         \end{cases}
\end{equation}
i.e., the set $\cM_{n,c}$ contains all matchings that have either the same weight or whose weights differ by~$n-2$.

We now establish explicit formulas for the cardinalities of the sets~$\cW_{n,c}$ and~$\cM_{n,c}$, proving a conjecture raised in~\cite{MR4046775} that expresses these quantities via \emph{generalized Narayana numbers $N_r(n,k)$}, defined as
\begin{equation}
\label{eq:narayana}
N_r(n,k) = \frac{r+1}{n+1} \binom{n+1}{k} \binom{n-r-1}{k-1}
\end{equation}
for any integers~$n\geq 1$, $r\geq 0$, and $1\leq k\leq n-r$; see~\cite{narayana-seq}.
The standard Narayana numbers are obtained for $r=0$.

\begin{theorem}
\label{thm:narayana}
For even $n\geq 2$ and any $c=0,1,\ldots,n-2$, we have $|\cW_{n,c}|=N_1(n,|c|+1)/2$.
\end{theorem}

By Lemma~\ref{lem:weight}, there are no centered flips between any matchings from two distinct sets $\cM_{n,c}$, $c=0,1,\ldots,n-2$, i.e., the cardinalities $|\cM_{n,c}|$ are an upper bound for the size of the components of the graph~$\cH_n$.
We can now compute these bounds explicitly.

\begin{corollary}
\label{cor:size}
For even $n\geq 2$, every component of~$\cH_n$ has at most $N_1(n,n/2)$ vertices.
Asymptotically, this is a $2/\sqrt{\pi n}(1+o(1))$-fraction of all vertices of the graph.
\end{corollary}

\begin{proof}
We first show that for even $n\geq 2$, the number~$N_1(n,k)$ is maximized for $k=n/2$.
Indeed, using the definition~\eqref{eq:narayana} and the definition of binomial coefficients, we obtain that $N_1(n,k+1)-N_1(n,k)$ is positive for all $k<x:=n/2-1/4-3/(4(2n+1))$ and negative for all $k>x$.
Note that $n/2-1<x<n/2$ for $n\geq 2$, and consequently, $N_1(n,k)$ is maximized for $k=n/2$.
Using Theorem~\ref{thm:narayana} and~\eqref{eq:Mmc}, it follows that $|\cM_{n,c}|$ is maximized when $c+1=n/2$, and the value of this quantity is $2N_1(n,n/2)/2=N_1(n,n/2)$ for this value of~$c$.
This proves the first part of the lemma.

To prove the second part, recall that the total number of vertices of~$\cH_n$ is $C_n$, the $n$th Catalan number.
By using the definition $C_n=\frac{1}{n+1}\binom{2n}{n}$ and~\eqref{eq:narayana}, the asymptotic value of the fraction $N_1(n,n/2)/C_n$ can be determined as
\begin{equation}
\label{eq:comp-frac}
\frac{N_1(n,n/2)}{C_n}=\frac{\frac{2}{n+1}\binom{n+1}{n/2}\binom{n-2}{n/2-1}}{\frac{1}{n+1}\binom{2n}{n}}=\frac{2\cdot\frac{n+1}{n/2+1}\binom{n}{n/2}\cdot\frac{(n/2)^2}{n(n-1)}\binom{n}{n/2}}{\binom{2n}{n}}=\frac{\binom{n}{n/2}^2}{\binom{2n}{n}}(1+o(1))
\end{equation}
Applying the standard estimate $\binom{2n}{n}=4^n/\sqrt{\pi n}(1+o(1))$ for the central binomial coefficients, which follows from Stirling's formula, on the right-hand side of~\eqref{eq:comp-frac} shows that this fraction equals $2/\sqrt{\pi n}(1+o(1))$, as claimed.
\end{proof}

\subsection{Proof of Theorem~\ref{thm:narayana}}
\label{sec:narayana}

The proof of Theorem~\ref{thm:narayana} is split into several lemmas.

A \emph{non-negative lattice path} is a path in the integer lattice~$\mathbb{Z}^2$ that starts at the origin, and that consist of \emph{upsteps} that change the current coordinate by~$(+1,+1)$ and of \emph{downsteps} that change the current coordinate by~$(+1,-1)$, and that never moves below the line~$y=0$; see Figure~\ref{fig:dyck}.
If the number of upsteps equals the number of downsteps, then such a lattice path is called a \emph{Dyck path}.
The set of all Dyck paths is denoted by~$D_n$, and their number is known to be $|D_n|=C_n$.
A \emph{peak} in a non-negative lattice path is an upstep immediately followed by a downstep.
Combinatorially, the generalized Narayana numbers~$N_r(n,k)$ count non-negative lattice paths with $n$ upsteps, $n-r$ downsteps, and exactly $k$~peaks.
In particular, $N_0(n,k)$ counts Dyck paths of length~$2n$ with exactly $k$~peaks, i.e., we have $C_n=\sum_{k=1}^n N_0(n,k)$.

We now consider another parameter for Dyck paths, introduced by Osborn~\cite{MR2732117}.
Given a Dyck path~$x\in D_n$, we consider each of the horizontal strips $\mathbb{R}\times [2i,2i+1]$ for $i=0,1,\ldots$, and we count how many upsteps of~$x$ lie in one of these strips.
The resulting count is the \emph{band-weight of~$x$}, denoted $w(x)$.
This definition is illustrated in the top part of Figure~\ref{fig:segment}.

\begin{lemma}[\cite{MR2732117}]
\label{lem:osborn}
For any $n\geq 1$ and $1\leq k\leq n$, there is a bijection~$\psi$ between Dyck paths from~$D_n$ with band-weight~$k$ and Dyck paths from~$D_n$ with~$n-k+1$ peaks.
Consequently, both sets are counted by the Narayana numbers~$N_0(n,n-k+1)=N_0(n,k)$.
\end{lemma}

\begin{figure}
\makebox[0cm]{ 
\includegraphics{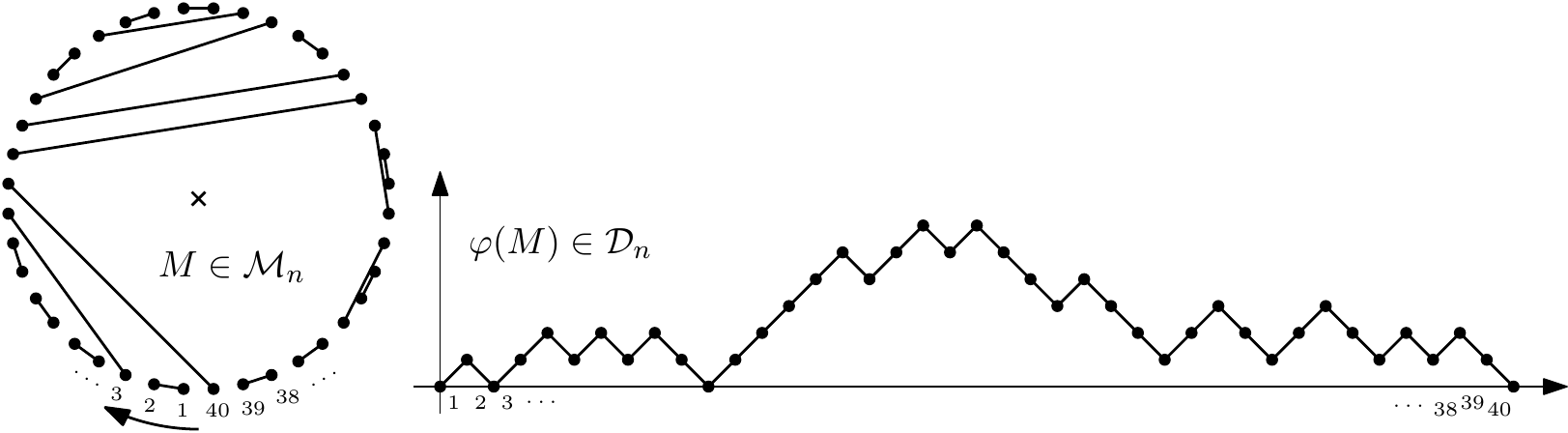}
}
\caption{Bijection between non-crossing matchings and Dyck paths defined after Lemma~\ref{lem:osborn}.
}
\label{fig:dyck}
\end{figure}

Given a matching $M\in\cM_n$, we define a Dyck path $\varphi(M)\in D_n$ as follows; see Figure~\ref{fig:dyck}:
We consider the points $1,\ldots,2n$ along the circle in increasing order, and we record an upstep or downstep for each point according to the following rule.
Every edge of~$M$ has two endpoints, and when we encounter the first of these two points, we record an upstep, and when we encounter the second of these two points, we record a downstep.
It is easy to see that $\varphi$ is a bijection between the sets~$\cM_n$ and~$D_n$.
Moreover, $\varphi$ maps perimeter edges to peaks of the Dyck path (except the edge $(1,2n)$).

Given a non-crossing matching $M\in\cM_n$ and a visible edge $e\in M$, the \emph{segment of~$M$ determined by the edge~$e$}, denoted $M_e$, is the set of all edges of~$M$ that are hidden behind~$e$, plus the edge~$e$ itself.
We also define $M_e^-:=M_e\setminus e$.
Clearly, $M_e$ is also a matching, and therefore also has a weight $w(M_e)=\sum_{f\in M_e}\sgn(f) \cdot \ell(f)$.
Moreover, the total weight~$w(M)$ is obtained by summing the weights~$w(M_e)$ over all visible edges~$e\in M$.

The following lemma asserts that the weight of a matching segment~$M_e$ is given by the band-weight of the corresponding Dyck path~$\varphi(M_e^-)$; see Figure~\ref{fig:segment}.
Note that the Dyck path~$\varphi(M_e^-)$ is constructed by applying the aforementioned bijection~$\varphi$ only to the edges in~$M_e^-$, i.e., this Dyck path has length~$2|M_e^-|$.

\begin{lemma}
\label{lem:segment}
Let $n\geq 2$ be even, and let $M\in\cM_n$ be matching with non-zero weight.
If $M$ has positive weight, then all visible edges of~$M$ are positive, and for every visible edge~$e\in M$ we have $w(M_e)=w(\varphi(M_e^-))$.
If $M$ has negative weight, then all visible edges of~$M$ are negative, and for every visible edge~$e\in M$ we have $w(M_e)=-w(\varphi(M_e^-))$.
\end{lemma}

\begin{figure}
\includegraphics{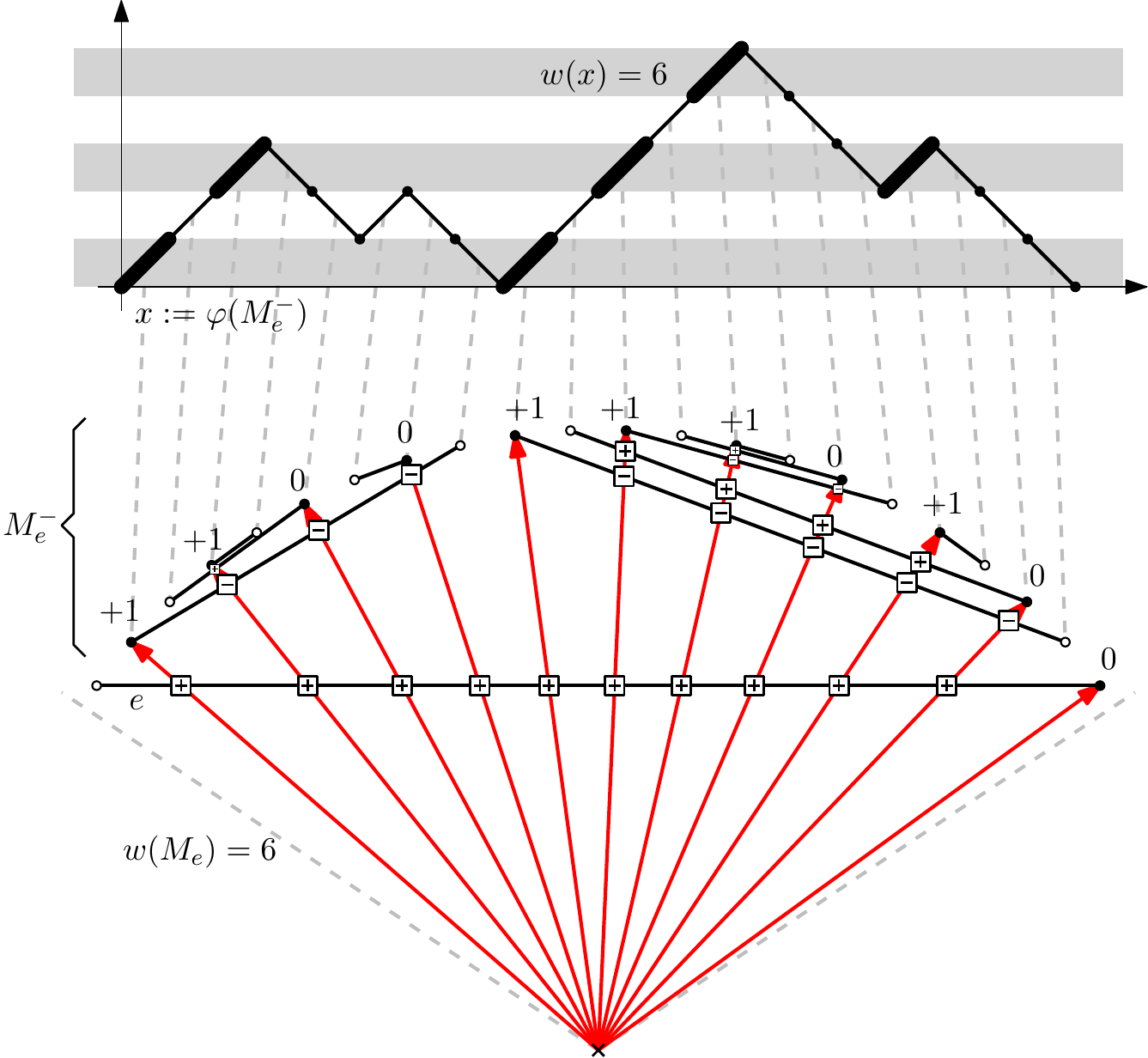}
\caption{Bijection between the segment of a non-crossing matching and a Dyck path, and the corresponding weights.
The bold steps on the Dyck path are the upsteps that contribute to its band-weight.
The crossings between rays and matching edges are marked by $+$ and $-$, showing the contributions of each crossing to the matching weight and the ray-weight.
The resulting ray-weights $w(r_k)\in\{0,+1\}$ are written at the endpoints of the rays.
}
\label{fig:segment}
\end{figure}

\begin{proof}
For any odd point~$k$ in our point set, let $r_k$ denote the ray from the circle center to that point, and define the \emph{ray-weight of~$r_k$}, denoted $w(r_k)$, as the number of all positive edges crossed by the ray~$r_k$ minus the number of all negative edges crossed by~$r_k$; see Figure~\ref{fig:segment}, where the edge matching the endpoint of the ray~$r_k$ does not count as crossed.
For any fixed edge~$e\in M$, the number of crossings of~$e$ with such rays equals~$\ell(e)$, and so we have
\begin{equation}
\label{eq:wMrk}
w(M)=\sum_{e\in M} \sgn(e) \cdot \ell(e)=\sum_{k=1,3,\ldots,2n-1} w(r_k).
\end{equation}

Observe that along any ray~$r_k$, the edges of~$M$ crossed by that ray are alternately positive and negative.
This is because the number of points that lie to the left or right of~$r_k$ and between the endpoints of two consecutively crossed edges of~$M_e$ is even.
It follows that $w(r_k)\in\{-1,0,+1\}$.
Also note that the visible edges in~$M$ must either all be positive, or all negative.
If they are all positive, then we have $w(r_k)\in\{0,+1\}$, and~\eqref{eq:wMrk} implies that~$w(M)>0$.
If they are all negative, then we have $w(r_k)\in\{-1,0\}$, and~\eqref{eq:wMrk} implies that~$w(M)<0$.

Now consider a matching~$M$ with positive weight, and one of its visible edges~$e\in M$.
We already know that the edge~$e$ is positive, and so every ray crossing this edge satisfies $w(r_k)\in\{0,+1\}$.
We also have
\begin{equation}
\label{eq:wMe}
w(M_e)=\sum_{f\in M_e} \sgn(f) \cdot \ell(f)=\sum_{k\in V(M_e)} w(r_k),
\end{equation}
where $V(M_e)$ denotes the set of all odd points matched by~$M_e$.
Clearly, we have $w(r_k)=+1$ if and only if~$r_k$ crosses an odd number of edges of~$M_e$.
Moreover, $w(r_k)=+1$ if and only if the edge from~$M_e$ that matches the point~$k$ lies to the right of the ray~$r_k$, as the number of points from~$V(M_e)$ to the left of~$r_k$ is odd.
Consequently, $w(r_k)=+1$ if and only if the step of the Dyck path~$\varphi(M_e^-)$ corresponding to the point~$k$ is an upstep, and this upstep lies in one of the horizontal strips~$\mathbb{R}\times[2i,2i+1]$ for $i=0,1,\ldots$.
From this and~\eqref{eq:wMe} we obtain $w(M_e)=w(\varphi(M_e^-))$.

By replacing $+1$ by $-1$ in the above argument, we obtain that if a matching~$M$ has negative weight, then each of its visible edges~$e\in M$ satisfies $w(M_e)=-w(\varphi(M_e^-))$.
This completes the proof of the lemma.
\end{proof}

For $c=2,\ldots,n$ we let $\cL_{n,c}$ denote the set of all matchings from~$\cM_n$ that have exactly $c$ perimeter edges.

\begin{lemma}
\label{lem:WL}
For even $n\geq 2$ and any $c=1,2,\ldots,n-2$ we have $|\cW_{n,c}\cup \cW_{n,-c}|=|\cL_{n,n-c}|$.
\end{lemma}

\begin{proof}
We establish a bijection between the sets $\cW_{n,c}\cup\cW_{n,-c}$ and~$\cL_{n,n-c}$.
Given a matching $M\in\cW_{n,c}$, let $E\seq M$ be the set of visible edges, and define $E_0:=\{e\in E\mid \ell(e)=0\}$ and $E_1:=E\setminus E_0$.
We partition the matching~$M$ into segments $M_e$, one for every visible edge~$e\in E$.
We know that $w(M)=c$, and so $\sum_{e\in E}w(M_e)=\sum_{e\in E_1}w(M_e)=c$.
For any edge~$e\in E_1$ and the corresponding matching $M_e$ we define $n_e:=|M_e^-|=|M_e|-1$.
As $w(M)=c>0$, Lemma~\ref{lem:segment} implies that all visible edges of~$M$ are positive and that the Dyck path~$\varphi(M_e^-)\in D_{n_e}$ has band-weight~$w(M_e)$ for all~$e\in E_1$.
Moreover, by Lemma~\ref{lem:osborn}, there is a Dyck path~$\psi(\varphi(M_e^-))\in D_{n_e}$ with exacly~$n_e-w(M_e)+1$ peaks.
Applying the inverse of the bijection~$\varphi$ to this Dyck path, we may replace the matching~$M_e^-$ by the matching $M_e':=\varphi^{-1}(\psi(\varphi(M_e^-)))$ on the same vertex set.
Moreover, the matching~$M_e'$ has exactly $n_e-w(M_e)+1$ many perimeter edges, as $\varphi^{-1}$ maps peaks to perimeter edges.
We let $M'$ denote the resulting matching, obtained by replacing every matching~$M_e$ by~$M_e'$ for all~$e\in E_1$.
Moreover, the number of perimeter edges of~$M'$ is
\begin{equation*}
|E_0|+\sum_{e\in E_1}(n_e-w(M_e)+1)=\underbrace{|E_0|+\sum_{e\in E_1}(n_e+1)}_{=n}-\underbrace{\sum_{e\in E_1}w(M_e)}_{=c}=n-c.
\end{equation*}
The mapping $M\mapsto M'$ can be extended analogously to matchings from~$\cW_{n,-c}$.
Note that $M$ and $M'$ have the same set of visible edges, and so the resulting mapping is indeed a bijection.
\end{proof}

\begin{lemma}
\label{lem:N1}
For even $n\geq 2$ and any $c=0,1,\ldots,n-2$, we have $|\cL_{n,n-c}|=N_1(n,c+1)$.
\end{lemma}

\begin{proof}
In this proof we will use the combinatorial interpretation of Narayana numbers discussed at the beginning of Section~\ref{sec:narayana}.
We will also use that
\begin{equation}
\label{eq:symm}
N_0(n,k)=N_0(n,n-k+1),
\end{equation}
which can be verified directly using the definition~\eqref{eq:narayana}.

Our goal is to count non-crossing matchings with $n$ edges and exactly $n-c$ many perimeter edges.
For this we partition the set $\cL_{n,n-c}$ into two sets $A$ and $B$, where $A$ contains all matchings that have an edge between points~1 and~$2n$, and $B$ contains all matchings that do not have an edge between these two points.
We first compute the size of~$A$.
For a matching~$M\in A$, consider the corresponding Dyck path~$x:=\varphi(M)$.
The Dyck path~$x$ starts with an upstep, ends with a downstep, and the middle $2(n-1)$ steps form a Dyck path from~$D_{n-1}$ that has exactly $n-c-1$ peaks.
Moreover, all Dyck paths of this form are obtained as images of~$A$ under~$\varphi$, and so by the aforementioned combinatorial interpretation of the Narayana numbers we obtain
\begin{equation}
\label{eq:sizeA}
|A|=N_0(n-1,n-c-1)\eqBy{eq:symm} N_0(n-1,c+1).
\end{equation}
We now compute the size of~$B$.
For a matching~$M\in B$, consider the corresponding Dyck path~$x:=\varphi(M)$.
It has $2n$ steps, exactly $n-c$ peaks, and it has the additional property that it touches the abscissa at least three times, once in the beginning, once in the end, and at least once more somewhere in between.
We count such paths by considering all Dyck paths with $2n$ steps and exactly $n-c$ peaks, and subtract the ones that touch the abscissa exactly twice (once in the beginning and once in the end).
The latter Dyck paths have the following form: they start with an upstep, end with a downstep, and the middle $2(n-1)$ steps form a Dyck path from~$D_{n-1}$ with exactly $n-c$ peaks.
Consequently, we obtain
\begin{equation}
\label{eq:sizeB}
|B|=N_0(n,n-c)-N_0(n-1,n-c)\eqBy{eq:symm} N_0(n,c+1)-N_0(n-1,c).
\end{equation}
Combining these observations yields
\begin{equation}
\label{eq:N0sum}
|\cL_{n,n-c}|=|A|+|B| \eqByM{\eqref{eq:sizeA},\eqref{eq:sizeB}} N_0(n-1,c+1)+N_0(n,c+1)-N_0(n-1,c).
\end{equation}
To complete the proof, we need to show that the right-hand side of~\eqref{eq:N0sum} equals~$N_1(n,c+1)$.

We know that $N_1(n,c+1)$ counts non-negative lattice paths with $n$ upsteps, $n-1$ downsteps and exactly $c+1$ peaks.
We partition these lattice paths into two sets~$A'$ and~$B'$, where $A'$ contains all those that end with an upstep, and $B'$ contains all those that end with a downstep.
We clearly have
\begin{equation}
\label{eq:N1AB'}
N_1(n,c+1)=|A'|+|B'|.
\end{equation}
To count the lattice paths in~$A'$, we remove the last upstep, yielding a Dyck path with $2(n-1)$ steps and exactly $c+1$ peaks, showing that
\begin{equation}
\label{eq:sizeA'}
|A'|=N_0(n-1,c+1).
\end{equation}
To count the lattice paths in~$B'$, we append a downstep in the end, yielding a Dyck path with $2n$ steps and exactly $c+1$ peaks, with the additional property that the last two steps are both downsteps.
We count such paths by considering all Dyck paths with $2n$ steps and exactly $c+1$ peaks, and subtract the ones that end with an upstep followed by a downstep.
Omitting the last two steps from the latter set of Dyck paths yields Dyck paths with $2(n-1)$ steps and exactly $c$ peaks.
We thus obtain
\begin{equation}
\label{eq:sizeB'}
|B'|=N_0(n,c+1)-N_0(n-1,c).
\end{equation}
Combining \eqref{eq:N1AB'}, \eqref{eq:sizeA'}, and~\eqref{eq:sizeB'} shows that the right-hand side of~\eqref{eq:N0sum} equals $N_1(n,c+1)$, as claimed.
\end{proof}

We are now ready to prove Theorem~\ref{thm:narayana}.

\begin{proof}[Proof of Theorem~\ref{thm:narayana}]
For $c=0$, we have $|\cW_{n,0}|=N_1(n,1)/2=1$.
For $c\geq 1$, combining Lemma~\ref{lem:WL} and~\ref{lem:N1} shows that $|\cW_{n,c}\cup \cW_{n,-c}|=N_1(n,c+1)$.
Furthermore, rotation by~$\pi/n$, either clockwise or counterclockwise, is a bijection between the sets~$\cW_{n,c}$ and~$\cW_{n,-c}$, proving that both sets have the same cardinality, so $|\cW_{n,c}|=|\cW_{n,-c}|=N_1(n,c+1)/2$.
This completes the proof of the theorem.
\end{proof}

\section{Rainbow cycles}
\label{sec:rainbow}

In this section, we briefly turn back our attention to the flip graph~$\cG_n$ discussed in the beginning, which contains all possible flips, not just the centered ones.

Recall that for any integer $r\geq 1$, an \emph{$r$-rainbow cycle} in the graph~$\cG_n$ is a cycle with the property that every possible matching edge appears exactly $r$ times in flips along this cycle.
Clearly, this means that every edge must also disappear exactly $r$ times in flips along the cycle.
The notion of rainbow cycles was introduced in~\cite{MR4046775}, and studied for several different flip graphs, including the graph~$\cG_n$.
The authors showed that $\cG_n$ has a 1-rainbow cycle for $n=2,4$, and a 2-rainbow cycle for $n=6,8$.
It was also proved that $\cG_n$ has no 1-rainbow cycle for any odd $n\geq 3$ and for $n=6,8,10$.
The last result was extended to the case $n=12$ in~\cite{Milich2018}.

In this section, we extend these results as follows:

\begin{theorem}
\label{thm:rainbow}
For odd $n\geq 3$ and any $r\geq 1$, there is no $r$-rainbow cycle in~$\cG_n$. \\
For even $n\geq 2$ and any $r>2/n^2\cdot N_1(n,n/2)$, there is no $r$-rainbow cycle in~$\cG_n$.
\end{theorem}

\begin{proof}
For any $n\geq 2$, the number of possible matching edges is~$n^2$.
As in every flip, two edges appear in the matching and two edges disappear, an $r$-rainbow cycle must have length~$rn^2/2$.

Let $n\geq 3$ be odd.
There are $2n$ distinct possible matching edges for each length $c=0,\ldots,\mu-1$, and only $n$ distinct possible matching edges of length~$c=\mu$.
It follows that the average length of all edges appearing or disappearing along an $r$-rainbow cycle is
\begin{equation*}
\frac{\sum\nolimits_{c=0}^{\mu-1}c\cdot 2n+\mu\cdot n}{n^2} \;\;\eqBy{eq:mu}\;\; \frac{n-2}{4}+\frac{1}{4n}>\frac{n-2}{4}.
\end{equation*}
However, by Lemma~\ref{lem:centered}, the average length of the four edges appearing or disappearing in a centered flip is only $(n-2)/4$, and even smaller for any non-centered flip.
Consequently, there can be no $r$-rainbow cycle.

Let $n\geq 2$ be even.
It was proved in~\cite[Lemma~10]{MR4046775} (with a similar averaging argument as given before for odd~$n$) that every $r$-rainbow cycle in~$\cG_n$ may only contain centered flips, i.e., we may restrict our attention to the subgraph $\cH_n\seq \cG_n$ given by centered flips.
By Corollary~\ref{cor:size}, all components of this graph contain at most $N_1(n,n/2)$ vertices.
Consequently, if the length of the cycle exceeds this bound, then no such cycle can exist.
This is the case if $rn^2/2>N_1(n,n/2)$, or equivalently, if $r>2/n^2\cdot N_1(n,n/2)$.
\end{proof}

\section{Open questions}
\label{sec:open}

We conclude this paper with some open questions.

\begin{itemize}[topsep=0mm,leftmargin=4mm]
\item
For odd $n\geq 3$, a natural task is to narrow down the bounds for the diameter of the graph~$\cH_n$ given by Theorem~\ref{thm:diam}.
We believe that the answer is $3n-7$, which is the correct value for $n=3,5,7,9,11$.
To this end, it seems worthwile to further investigate the layer structure of the graph~$\cH_n$ for odd~$n$ (see Figure~\ref{fig:h5}), where the two matchings that have only perimeter edges appear at the left and right end.
In particular, what is the combinatorial interpretation of the $3n-6=3(n-2)$ layers of this graph?

\item
For even $n\geq 4$, it would be very interesting to prove that the number of components of the graph~$\cH_n$ is exactly $C_{n/2}+n-3$, which we established as a lower bound in Corollary~\ref{cor:comp}.
We verified with computer help that this lower bound is tight for $n=4,6,8,10,12,14$.
What remains to be shown is that matchings that are not centrally symmetric and that either have the same weight or weight difference~$(n-2)$ can be transformed into one another by centered flips.
The proof of Lemma~\ref{lem:not-tree} might give an idea how to achieve this.

\item
It is also open whether $r$-rainbow cycles exist in the graph~$\cG_n$ for even $n\geq 14$ and any~$1\leq r\leq 2/n^2\cdot N_1(n,n/2)$.
As mentioned before, we may restrict our search to the subgraph~$\cH_n\seq \cG_n$.

\end{itemize}

\section*{Acknowledgements}

We thank the anonymous reviewers of this paper who provided many insightful comments.
In particular, one referee's observation about our proof of Lemma~\ref{lem:4flips} improved our previous upper bound on the diameter of~$\cH_n$ for odd~$n$ from $\cO(n\log n)$ to~$\cO(n)$ (recall Theorem~\ref{thm:diam}).

Figure~\ref{fig:h6} in our paper was obtained by slightly modifying Figure~10 from~\cite{MR4046775}, and the authors of this paper kindly provided us with the source code of their figure.

\bibliographystyle{alpha}
\bibliography{refs}

\end{document}